\documentclass{amsart}
\usepackage{amssymb,
amsmath,
amsthm,
graphicx,
amscd,
multind,
eufrak,
hyperref,
}
\theoremstyle{plain}
\newtheorem{thm}{Theorem}[section]
\newtheorem{lm}[thm]{Lemma}
\newtheorem{prop}[thm]{Proposition}
\newtheorem{prb}[thm]{Problem}
\theoremstyle{definition}
\newtheorem{ex}[thm]{Example}
\newtheorem{re}[thm]{Remark}

\newcommand{\RR}{{\mathbb R}}
\newcommand{\ZZ}{{\mathbb Z}}
\newcommand{\NN}{{\mathbb N}}
\newcommand{\PP}{{\mathbb P}}
\renewcommand{\AA}{{\mathbb A}}
\newcommand{\Rb}{\RR_\infty}
\newcommand{\Trop}{\operatorname{Trop}}
\newcommand{\trop}{\operatorname{trop}}
{\begin{figure} \begin{center}}%
{\end{center} \end{figure}}

\newcommand{\id}{\operatorname{id}}

\newcommand{\ini}{\mathrm{in}}

\newcommand{\one}{{\mathbf 1}}

\usepackage{xypic}
\newcommand{\an}{\mathrm{an}}
\newcommand{\Xan}{X^{\an}}
\newcommand{\Yan}{Y^{\an}}
\newcommand{\Gm}{\mathbb{G}_{\mathrm{m}}}
\renewcommand{\phi}{\varphi}

\newcommand{\A}{{\mathbb A}}
\newcommand{\G}{{\mathbb G}}
\newcommand{\R}{{\mathbb R}}
\renewcommand{\hat}[1]{\widehat{#1}}
\newcommand{\cGr}{\widehat{\mathrm{Gr}}}
\newcommand{\im}{\mathrm{im}}

\begin{document}

\title{Faithful tropicalisation and torus actions}

\author[J.~Draisma]{Jan Draisma}
\address[Jan Draisma]{
Department of Mathematics and Computer Science\\
Technische Universiteit Eindhoven\\
P.O. Box 513, 5600 MB Eindhoven, The Netherlands;
and Centrum voor Wiskunde en Informatica, Amsterdam,
The Netherlands}
\thanks{JD is supported by a Vidi grant from
the Netherlands Organisation for Scientific Research (NWO)}
\email{j.draisma@tue.nl}

\author[E.~Postinghel]{Elisa Postinghel}
\address[Elisa Postinghel]{
Department of Mathematics\\
Katholieke Universiteit Leuven\\
Celestijnenlaan 200b - box 2400, 3001 Leuven, Belgium
}
\thanks{EP is supported by the Research Foundation - Flanders (FWO)}
\email{elisa.postinghel@wis.kuleuven.be}

\maketitle

\begin{abstract}
For any affine variety equipped with coordinates, there is a surjective,
continuous map from its Berkovich space to its tropicalisation.
Exploiting torus actions, we develop techniques for finding an explicit,
continuous section of this map. In particular, we prove that such a
section exists for linear spaces, Grassmannians of planes (reproving a
result due to Cueto, H\"abich, and Werner), matrix varieties defined by
the vanishing of $3 \times 3$-minors, and for the hypersurface defined
by Cayley's hyperdeterminant.

\end{abstract}

\section{Introduction} \label{sec:Intro}
Let $K$ be a field with a non-Archimedean valuation $v:K \to
\Rb:=\RR \cup \{\infty\}$, let $\AA^n \supseteq \Gm^n$ be the
$n$-dimensional affine space over $K$ and the $n$-dimensional torus with
coordinates $x_1,\ldots,x_n$, respectively, and let $\PP^{n-1}$ be the
$(n-1)$-dimensional projective space over $K$ with
homogeneous coordinates $x_1,\ldots,x_n$. For a closed subvariety
$X$ of $\Gm^n$ or $\AA^n$ or $\PP^{n-1}$, defined over $K$, we write $\Trop(X)$ for
the tropicalisation of $X$ sitting inside $\RR^n$ or $\Rb^n$ or $(\Rb^n \setminus
\{(\infty,\ldots,\infty)\})/\RR(1,\ldots,1)$, respectively.

Write $\Xan$ for the analytification of $X$ in Berkovich's sense
\cite[Chapter 1]{Berkovich90}. We work with the negative logarithms
of multiplicative seminorms, so in the affine case $\Xan$ is the
set of all ring valuations $K[X] \to \Rb$ extending $v$, equipped
with the topology of pointwise convergence. In particular, $\Xan$
is a Hausdorff topological space, and a sequence of $w_1,w_2,\ldots$
in $\Xan$ converges if and only if the sequence $w_1(f),w_2(f),\ldots$
converges in $\Rb$ (with the topology of a half-open
interval) for each $f \in K[X]$.  Write $\infty$ for the valuation of
$K[\AA^n]=K[x_1,\ldots,x_n]$ that maps a polynomial to the valuation of
its constant term. In the projective case, let $\hat{X} \subseteq \AA^n$
be the affine cone over $X$. Then, as a topological space, $\Xan$ equals
$\hat{X}^{\an} \setminus \{\infty\}$ modulo the equivalence relation
under which $w_1$ and $w_2$ are equivalent if and only if there exists a
constant $C \in \RR$ such that for each degree-$d$ homogeneous polynomial
$f$ in the graded coordinate ring $K[\hat{X}]$ we have $w_1(f)=dC+w_2(f)$
(see \cite[Chapter 2]{Baker10} for the case of the projective line).

There is a continuous surjection 
\[ \trop:\Xan \to \Trop(X),\ w \mapsto (w(x_1),\ldots,w(x_n)). \] 
This can be taken either as a definition of $\Trop(X)$
or as a theorem when other de\-fi\-nitions are chosen
\cite{Einsiedler06,Payne09,Payne07,Draisma06a}. Indeed, in \cite{Payne09} it
is proved that $\Xan$ is the projective limit of the tropicalisations
$\Trop(X)$ for all choices of coordinates. The tropicalisation is the
support of a finite polyhedral complex by \cite{Bieri84}.

In this paper we discuss a number of high-dimensional examples where
$\trop$ has a continuous section. The results are motivated by exciting
recent work for Grassmannians of planes \cite{Cueto14} and for curves
\cite{Baker11}. In particular, we will give another, more geometric
proof of the main result of \cite{Cueto14} that Grassmannians of planes
admit such a section. In the very recent paper \cite{Gubler14} (written
concurrently with our paper) it is proved that, if $X$ is a subvariety
of $\Gm^n$, then a section exists on the locus in $\Trop(X)$ where
the tropical multiplicity equals one \cite{Gubler14}. This beautiful
general theorem implies parts of our results, e.g. for linear spaces.
The emphasis in our paper, however, is on explicit sections in concrete
examples, and in several of these we also extend the section to the
part of $\Trop(X)$ outside $\RR^n$. 

Throughout, we will assume that the valuation $K \to \Rb$ is
surjective. This is no restriction for our purposes. Indeed, $(K,v)$
always embeds into a valued field $(L,v_L)$ with $v_L$ surjective. This
does not change $\Trop(X)$, and a suitable section $\Trop(X) \to\Xan_L$
can be composed with the restriction map $\Xan_L \to \Xan$ to obtain a
section $\Trop(X) \to \Xan$.

We will use the following notation and facts. Given a point $\xi \in
\Rb^n$ we write
\[ K[x]_\xi:=\left\{\sum_{\alpha \in A} c_\alpha x^\alpha \mid 
A \subseteq \NN^n \text{ finite and } v(c_\alpha) + \alpha
\cdot \xi \geq 0 \text{ for all } \alpha \in A\right\} \]
for the {\em tilted group ring} \cite{Payne07}. This ring contains the
valuation ring of $K$ and it has an ideal with the same definition but
with $\geq$ replaced by $>$. The quotient by this ideal is an algebra
over the residue field $k$ of $K$. By surjectivity of the valuation,
this algebra is in fact a polynomial ring over $k$ in at most $n$
variables---generators can be obtained as the images $y_i$ of $c_i x_i$
where $i$ ranges throught the set where $\xi_i \neq \infty$ and where
the coefficients $c_i \in K$ are chosen such that $v(c_i)+\xi_i=0$. Let
$I(X)$ be the ideal of $X$ in $K[x]$.  The image of $I(X) \cap K[x]_\xi$
in the polynomial ring $k[y_i \mid i:\xi_i \neq \infty]$ is called the
{\em initial ideal} $\ini_\xi I(X)$ of $I(X)$ relative to $\xi$, and the
scheme $\ini_\xi X$ defined by it is called the {\em initial degeneration}
of $X$. The point $\xi$ lies in the tropical variety if and only if
$\ini_\xi I(X)$ does not contain monomials \cite[Chapter 3]{Maclagan12}.

The remainder of this paper is organised as follows. In
Section~\ref{sec:Linear} we prove that if $Y \subseteq \AA^n$ is a
linear space, then the surjection $\Yan \to \Trop(Y)$ has a continuous
section. In Section~\ref{sec:Torus}, given an action of an $m$-dimensional
subtorus of $\Gm^n$ on a subvariety $X \subseteq \AA^n$, we construct an
action of $\RR^m$ on a retract $Z \subseteq \Xan$, which maps surjectively
and $\RR^m$-equivariantly onto $\Trop(X)$. In Section~\ref{sec:Smearing}
we introduce techniques for finding sections $\Trop(X) \to Z$ when
$X$ is obtained by smearing around a linear space $Y$ with a torus
action. As an example, we treat the variety in $\Gm^{m \times p}$
of matrices of less than full rank, where we show that a continuous
section exists at least over a large open subset of the tropicalisation.
In Sections~\ref{sec:Grassmannian} and~\ref{sec:RankTwo} we apply our
techniques to Grassmannians of two-spaces and to matrices of rank two,
respectively. We conclude with a brief discussion of $A$-discriminants
in Section~\ref{sec:ADiscriminants}.

\subsection*{Acknowledgments}
We thank Joe Rabinoff for interesting discussions and, more specifically,
for pointing out that the retract $Z$ of $\Xan$ that we define in
Section~\ref{sec:Torus} is, in fact, a strong deformation retract. We
also thank a referee for pointing out an error in an earlier version of
this paper concerning the variety of matrices of less than full rank.

\section{Linear spaces} \label{sec:Linear}

In this section we assume that $Y$ is a linear subspace through the
origin $0 \in \AA^n$. Tropical linear spaces are well-understood through
their circuits and cocircuits \cite{Yu07,Ardila06}, and the proof of
the following theorem is very natural from that perspective.

\begin{thm} \label{thm:Linear}
For any linear subspace $Y \subseteq \AA^n$ the projection
$\trop_Y: \Yan \to \Trop(Y)$ has a continuous section.
\end{thm}

Without loss of generality, we may restrict to the case where $Y$ is
not contained in any coordinate hyperplane, so that $\Trop(Y)$ is the
closure of $\Trop(Y) \cap \RR^n$. Nevertheless, we will need to check
carefully that the section we construct is also continuous on $\Trop(Y)
\setminus \RR^n$. We will use that for $\eta \in \Trop(Y) \cap \RR^n$
the initial degeneration $\ini_\eta Y$ is a linear subspace of $\AA_k^n$
of the same dimension as $Y$. For general $\eta \in \Trop(Y)$ we have
$\ini_\eta Y=\ini_\eta Y'$, where $Y'$ is the subspace of $Y$ consisting
of all $y$ with $x_i(y)=0$ for all $i$ with $\eta_i=\infty$.

The $K$-space $Y'$ defines a matroid on $[n]$ by declaring a subset $J
\subseteq [n]$ independent if the restrictions $(x_j)|_{Y'},\ j \in J$
are $K$-linearly independent. Similarly, the $k$-space $\ini_\eta Y'$
also defines a matroid on $[n]$, by declaring $J$ independent if the
restrictions of the $y_j, j \in J$ (from the definition of the tilted
polynomial ring) to $\ini_\eta Y'$ are $k$-linearly independent. The two
matroids have the same rank, and any basis of the latter matroid is also
a basis of the former matroid. Throughout the paper, these distinguished
bases of the former matroid will be called {\em compatible with $\eta$}
(and, conversely, $\eta$ with those bases).

\begin{proof}[Proof of Theorem~\ref{thm:Linear}.]
We define the section $\sigma:\Trop(Y) \to \Yan$ as follows. Pick
$\eta \in \Trop(Y)$, set $S:=\{i \in [n] \mid \eta_i=\infty\}$, and
let $Y' \subseteq Y$ be the subspace of all $y \in Y$ with $x_i(y)=0$
for all $i \in S$. Let $J$ be a basis of the matroid defined by $Y'$
that is compatible with $\eta$. In particular, $J$ is disjoint from $S$.
The inclusion $K[x_j \mid j \in J] \to K[Y']$ is an isomorphism, so for
$f \in K[Y]$ we can uniquely write $f|_{Y'}=\sum_\alpha c_\alpha x^\alpha$
where the $\alpha$ run through $\NN^J$. We set
\[ \sigma(\eta)(f):=\min_{\alpha}(v(c_\alpha) + \alpha \cdot
\eta). \]
This is clearly a valuation that maps $x_j, j \in J$ to $\eta_j$ and
that maps the $x_i$ with $i \in S$ to $\infty$. What about $x_i$ with $i \not \in J \cup S$? Up to
a scalar factor, there exists a unique non-zero linear relation
\[ \sum_{j \in J \cup \{i\}} d_j x_j  \in I(Y'). \]
After scaling we may assume that $v(d_j) + \eta_j \geq
0$ for all $j \in J \cup \{i\}$ and that equality holds for at least one
$j$. Then this element lies in $I(Y') \cap K[x_j \mid j \not
\in S]_\eta$. If $v(d_i)+\eta_i$
were strictly positive, then projecting down into $k[y_j
\mid j \not \in S]$ would yield a relation among the $y_j$ with $j \in J$, a contradiction
to the choice of $J$. Hence $v(d_i)+\eta_i=0$. If $v(d_j)+\eta_j$ were
strictly positive for all $j \in J$, then projecting down would yield
$y_i \in \ini_\eta I(Y')$, which contradicts $\eta \in
\Trop(Y')$. Hence
$v(d_j)+\eta_j=0$ for some $j \in J$. This shows that
\[ \sigma(\eta)(x_i)=\min_{j \in J} (v(-d_j/d_i) + \eta_j) =
\eta_i, \]
as required. So $\sigma(\eta) \in \Yan$ is a point in the fibre of $\trop_Y$
above $\eta$.

To define $\sigma(\eta)$, we have made the choice of a basis $J$ in
the matroid defined by $\ini_\eta I(Y')$. But in fact, this choice
does not influence the outcome. Indeed, {\em any} valuation $w \in
\Yan$ with $\trop_Y(w)=\eta$ must satisfy $w(f) \geq \sigma(\eta)(f)$
for all $f \in K[Y]$\footnote{So $\sigma(\eta)$ is the {\em
Shilov boundary point} \cite[Chapter 2]{Berkovich90} in the fibre in
$\Yan$ above $\eta$.}. In particular, this must hold for all valuations
constructed from other bases of the matroid.  This shows that $\sigma$
is well-defined on all of $\Trop(Y)$.

It remains to show that $\sigma$ is continuous. This is immediate from the
formula for $\sigma(\eta)$ on a subset of $\Trop(Y)$ where $S$ and $J$
compatible with $\eta$ are fixed. Let $Y'$ be as above. Suppose that a
sequence $\eta^{(l)},l=1,2,\ldots$ in this set converges to a point $\eta
\in \Trop(Y)$. Note that the set of $i$ with $\eta_i=\infty$ contains
$S$ but may be strictly larger, and may even contain elements of
$J$. Even so, for every non-zero relation
$\sum_{j \in J \cup \{i\}} d_j x_j \in I(Y')$ for $i \not \in J \cup S$
we have $\min_{j \in J} (v(d_j)+\eta^{(l)}_j) = v(d_i)+\eta^{(l)}_i$. This
closed condition then also holds in the limit:
\begin{equation} \label{eq:min}
\min_{j \in J}
(v(d_j)+\eta_j)=v(d_i)+\eta_i.
\end{equation} 
Let $w$ be the valuation of $K[Y]$
defined by mapping $f \in K[Y]$ with $f|_{Y'}=\sum_{\alpha \in \NN^J}
c_\alpha x^\alpha$ to
$\min_{\alpha \in \NN^J} (v(c_\alpha)+\alpha \cdot \eta)$.
Then $w$ is, indeed, a valuation of $K[Y]$, which maps $x_j$ to $\eta_j$
for $j \in J$ (because $x_j|_{Y'}$ is a single term) and for $j \in
S$ (because $x_j|_{Y'}$ has no terms) and for $j \not \in J \cup S$
(by~\eqref{eq:min}). Moreover, $w(f)$ is minimal among all such
valuations, so $w(f)=\sigma(\eta)(f)$. This shows that $\sigma$ is
continuous on the closure of the set of all $\eta$ compatible with a given
$S$ and $J$. These closures form a finite closed cover of $\Trop(Y)$,
hence $\sigma$ is continuous everywhere.
\end{proof}

\begin{re} \label{re:constant}
In the {\em constant coefficient case}, where $Y$ is a linear space
defined over a subfield of $K$ on which the valuation is trivial, 
the choice of $J$ above can be made more constructive, as follows.
Given $\eta \in \Trop(Y) \cap \RR^n$, take a permutation $\pi \in S_n$ such that
$\eta_{\pi(1)} \geq \ldots \geq \eta_{\pi(n)}.$
Then construct $J$ by setting $J_0:=\emptyset$ and 
\[ J_{i} := \begin{cases}
J_{i-1} \cup \{i\} &\text{ if $x_i|_Y$ linearly
independent of $\langle x_j|_Y \mid j \in J_{i-1} \rangle$,
and}\\
J_{i-1} &\text{ otherwise.}
\end{cases}
\]
Then $J:=J_n$ is a basis of the matroid defined by $Y$ compatible with
$\eta$. This is the greedy algorithm for finding a maximal-weight basis
in a matroid \cite[Chapter 40]{Schrijver03}.

Conversely, given a basis $J$ of that matroid, we can construct all $\eta
\in \Trop(Y) \cap \RR^n$ compatible with it by choosing the $\eta_j$
with $j \in J$ arbitrarily and setting $\eta_i$ for $i \not \in J$
equal to the minimal value $\eta_j$ for $j$ in the unique circuit
contained in $J \cup \{i\}$.  We will use this explicit construction in
Sections~\ref{sec:Grassmannian} and~\ref{sec:RankTwo}. These
remarks apply, {\em mutatis mutandis}, also to $\eta \in
\Trop(Y) \setminus \RR^n$. 
\end{re}

\section{Torus actions} \label{sec:Torus}

Let $\phi:\Gm^m \to \Gm^n$ be a homomorphism of tori. This is of the form
$\phi(t_1,\ldots,t_m)=(t^{a_1},\ldots,t^{a_n})$ where $a_1,\ldots,a_n
\in \ZZ^m$. Let $A \in \ZZ^{n \times m}$ be the matrix with rows
$a_1,\ldots,a_m$. Let $X \subseteq \AA^n$ or $X \subseteq \Gm^n$ be
a closed affine subvariety stable under the $\Gm^m$-action on $\AA^n$
(or $\Gm^n$) given by $\phi$. Then $\RR^m$ has a continuous action on
$\Trop(X)$ given by
\[ (\tau,\xi) \mapsto A\tau + \xi. \]
The column space of $A$ is contained in the {\em lineality space} of
$\Trop(X)$. In this section we investigate to what extent this action
can be lifted to $\Xan$. For this, we denote by $\lambda: \Gm^m \times
X \to X, (t,x) \mapsto \phi(t)x$ the action of $\Gm^m$ on $X$ and by
$\lambda^*:K[X] \to K[\Gm^m \times X]$ its comorphism.
\begin{lm}
There exists a commutative diagram of continuous maps:
\[ 
\xymatrixcolsep{5pc}
\xymatrix{
\RR^m \times \Xan \ar[r] \ar[d]^{\id \times \trop}  & (\Gm^m \times X)^\an
\ar[r]_{w \mapsto w \circ \lambda^* } & \Xan \ar[d]_\trop\\
\RR^m \times \Trop(X) \ar[rr]^{(\tau,\xi) \mapsto A \tau + \xi
} &&  \Trop(X).
}
\]
\end{lm}

\begin{proof}
The right-most map in the top row of the diagram is the analytification
of the torus action, hence in particular continuous. The only map that
needs a definition is the left-most map in that row. It sends $(\tau,w)$
to the valuation of $K[\Gm^m \times X]$ defined by
\[ \sum_{\beta \in \ZZ^m} f_\beta t^\beta \mapsto
\min_\beta (w(f_\beta) + \beta \cdot \tau). \]
For each fixed element $\sum_{\beta \in \ZZ^m} f_\beta t^\beta$ of
$K[\Gm^m \times X]$ the right-hand side is continuous in $(\tau,w)$
(this uses the definition of the topology of $\Xan$ and the fact that a
point-wise minimum of continuous functions is continuous).  By definition
of the topology of $(\Gm^m \times X)^\an$, this implies that the map
is continuous.

To see that the diagram commutes, pick $(\tau,w) \in \RR^m \times \Xan$
and let $w' \in \Xan$ be the image of that pair along the top row. We have
$\lambda^* x_i=t^{a_i} x_i$, and hence 
\[ w'(x_i)=w(x_i) + a_i \cdot \tau. \]
This implies that $\trop(w')=\trop(w)+A \tau$, as claimed.
\end{proof}

Let $\mu$ denote the composition of the two maps in the first
row. Unwinding the definitions, we find that $\mu$ sends $(\tau,w)$ to a
valuation on $K[X]$ defined as follows. Pick $f \in K[X]$ and decompose
$f=\sum_{\beta \in \ZZ^m} f_\beta$ into $\Gm^m$-weight vectors, i.e.,
with $\lambda^* f_\beta=t^\beta f_\beta$. Then
\[ \mu(\tau,w)(f)=\min_{\beta} (w(f_\beta)+ \beta \cdot \tau). \]
We remark that if $A\tau=0$, then $\mu(\tau,w)=\mu(0,w)$. Indeed, then
$\tau$ is perpendicular to the rows of $A$, hence to any $\ZZ$-linear
combination of these, and the $\beta$ for which there exist non-zero
$f_\beta \in K[X]$ of weight $\beta$ are such linear combinations.

In general, $\mu$ is not an action of $\RR^m$ on $\Xan$. Indeed, while
the valuations $\mu(0,w)$ and $w$ do agree on monomials, they do not
need to agree on other functions. For an explicit example, set $X=\AA^2$
with coordinate ring $K[x_1,x_2]$, let $m=2$, and let $\phi$
be the identity. Define $w \in \Xan$ by $w(f):=v(f(1,1))$, so that
$w(x_1)=w(x_2)=0$. Then the image of $(0,w)$ along the first row equals
the ``Gauss point'' $w'$ of $K[x_1,x_2]$ defined by
\[ \sum_{i,j} c_{ij} x_1^i x_2^j \mapsto \min_{i,j} v(c_{ij}). \]
Then we have $w(x_1-1)=\infty\neq 0=w'(x_1-1)$. However, the following
lemma shows that $\mu(0,w) \neq w$ is the {\em only} obstacle to $\mu$
being an action.

\begin{lm}
Define $Z$ as the image of $\mu$. Then $Z$ is a closed subset of $\Xan$
and the restriction of $\mu$ to $\RR^m \times Z$ defines a continuous
action of $\RR^m$ on $Z$. Moreover, the map $w \mapsto \mu(0,w)$ defines
a continuous retraction from $\Xan$ to $Z$.
\end{lm}

\begin{proof}
First, for $\tau_1,\tau_2 \in \RR^m$ and $f \in K[X]$ with $\Gm^m$-weight
decomposition $f=\sum_{\beta \in \ZZ^m} f_\beta \in K[X]$ and $w \in \Xan$
we compute
\begin{align*} 
&\mu(\tau_1,\mu(\tau_2,w))(f)=
\min_{\beta \in \ZZ^m} (\beta \cdot \tau_1 + \mu(\tau_2,w)(f_\beta))
\\&= 
\min_{\beta \in \ZZ^m} (\beta \cdot \tau_1 + \beta \cdot \tau_2 + w (f_\beta)) 
= \mu(\tau_1+\tau_2,w)(f).
\end{align*}
This implies that $\mu(0,\mu(\tau,w))=\mu(\tau,w)$, so that $0$ acts
as the identity on $Z$. Hence $\mu$ is an action on $Z$.  Furthermore,
$Z$ can be characterised as the pre-image of the diagonal in $\Xan
\times \Xan$ under the continuous map $\Xan \to \Xan \times \Xan, w \mapsto
(w,\mu(0,w))$. Since $\Xan$ is Hausdorff, the diagonal is closed, hence
so is $Z$.  The last statement is immediate.
\end{proof}

The following refinement of the statement that $Z$ is a
retract of $\Xan$ was pointed out to us by Joe Rabinoff.

\begin{prop}
In the setting above, $Z$ is a strong deformation retract of
$\Xan$.
\end{prop}

\begin{proof}
This can be derived using the general techniques of \cite[Chapter
6]{Berkovich90}; here is a shortcut in our language. For $r \in
[0,\infty]$ and $w \in \Xan$ let $w_r$ be the function $K[X] \to \Rb$
defined as follows. Take $f \in K[X]$, expand $f(\phi(t)x):=\sum_{\beta
\in \ZZ^m} f_\beta t^\beta$, and rewrite this Laurent series with
$K[X]$-coefficients as a formal power series
\[ \sum_{\beta \in \ZZ^m} f_\beta t^\beta =
        \sum_{\gamma \in (\ZZ_{\geq 0})^m} g_\gamma
        (t-1)^\gamma \]
around the identity element $1=(1,\ldots,1)$ of $\Gm^m$. Set
\[ w_r(f):=\min_\gamma (w(g_\gamma)+|\gamma|r), \text{
where } |\gamma|:=\gamma_1+\ldots+\gamma_m. \]
We argue that this minimum is attained, and that it can be replaced
by a minimum over a finite set of $\gamma$s that does not depend on $w$
or $r$.  In the rewriting process, we replace each Laurent monomial $t^\beta$
by the formal power series of $((t-1)+1)^\beta$ around $1$. This shows
that each $g_\gamma$ is a $\ZZ$-linear combination of the $f_\beta$. In
particular, for all $\gamma$ we have $w(g_\gamma) \geq \min_\beta
w(f_\beta)$, and for $r>0$ this suffices to conclude that the minimum
is attained.

Conversely, we claim that each $f_\beta$ is a $\ZZ$-linear combination
of the $g_\gamma$. This is immediate if all $\beta$ with $f_\beta \neq
0$ are already in $(\ZZ_{\geq 0})^m$ (since then we are just rewriting
polynomials, and the rewriting can be reversed). The general case can be
reduced to this, since multiplication of power series with a fixed power
series of the form $((t-1)+1)^\beta$ is a $\ZZ$-linear isomorphism
with inverse equal to multiplication with $((t-1)+1)^{-\beta}$.
Consequently, we find that the minimum is attained for $r=0$, as well,
and that $\min_\gamma w(g_\gamma)=\min_\beta w(f_\beta)=\mu(0,w)(f)$.

Combining the two $\ZZ$-linear transitions, all countably many $g_\gamma$
are $\ZZ$-linear combinations of finitely many among them. If $d$ is the maximum
value of $|\gamma|$ among these finitely many, then we can replace the
minimum defining $w_r(f)$ by the minimum over all $\gamma$ with $|\gamma|
\leq d$. Then it is evident that $w_r(f)$ depends continuously on the
pair $(w,r) \in \Xan \times [0,\infty]$.

Now $w_r$ is a point in $\Xan$ that 
depends continuously on $(w,r)$. For $r=\infty$ we have
\[ w_\infty(f)=w(g_0)=w(\sum_\beta f_\beta)=w(f), \] 
so $w_\infty=w$. As mentioned above, we have $w_0=\mu(0,w)$. Finally, we must argue that if $w$ already
lies in $Z$, that is, if $w=\mu(0,w)$, then $w_r=w$ for all $r \in
(0,\infty]$. But in this case the $\gamma=0$ term in the definition of
$w_r$ equals $\min_\beta w(f_\beta)$ and all other terms are (strictly)
larger than this, so that $w_r=w$ as desired.
\end{proof}

We conclude this section with two remarks on quotients. The first concerns
the categorical quotient $X//\Gm^m$ of $X$ by the action of $\Gm^m$, i.e.,
the affine variety with coordinate ring equal to the ring of $\Gm^m$-invariants
in $K[X]$. The morphism $X \to X//\Gm^m$ gives rise to a morphism of
analytic spaces, which sends a valuation $w \in \Xan$ to its restriction
to the $\Gm^m$-invariants.

\begin{lm} \label{lm:quotient}
The map $\Xan \to (X//\Gm^m)^\an$ factorises as 
\[ 
\Xan \to Z \to Z/\RR^m \to (X//\Gm^m)^\an. 
\]
\end{lm}

\begin{proof}
We need to show that, for $\tau \in \RR^m$ and $w \in \Xan$, the
restriction of $w':=\mu(\tau,w)$ to the $\Gm^m$-invariants $f \in
K[X]$ does not depend on $\tau$ and equals the restriction of $w$ to
$\Gm^m$-invariants. But this is immediate: $f$ has weight zero, and hence
\[ w'(f)=\mu(\tau,w)(f)=w(f)+0\cdot \tau=w(f), \]
as desired. 
\end{proof}

The second remark concerns the passage from affine cones to projective
varieties. Suppose that $X \subseteq \AA^n$ is an affine cone, and denote
by $\PP X \subseteq \PP^{n-1}$ the corresponding projective variety.
The points of $(\PP X)^\an$ are equivalence classes of points of $\Xan
\setminus\{\infty\}$.

\begin{lm} \label{lm:projective}
The map $Z \to (\PP X)^\an$ factorises as
\[ 
Z \to Z/U \to (\PP X)^\an,
\]
where $U:=A^{-1}\RR(1,\ldots,1)$. 
\end{lm}

\begin{proof}
We need to show that if $A\tau=(C,\ldots,C)$ for some $C \in \RR$ and
if $w \in Z$, then $w':=\mu(\tau,w)$ is equivalent to $w$. Thus let $f$
be a homogeneous polynomial of degree $d$ in the graded ring $K[X]$, and
decompose $f=\sum_{\beta \in \ZZ^m} f_\beta$. Then $\beta \cdot \tau=dC$
for all $\beta$ with $f_\beta$ non-zero, and hence
\[ w'(f)=\min_\beta (w(f_\beta)+\beta \cdot \tau) 
=dC+\min_\beta w(f_\beta)=\mu(0,w)=w. \]
\end{proof}

\section{Smearing a subspace around by a torus} \label{sec:Smearing}

Let $Y \subseteq \AA^n$ be a linear subspace not contained in any
coordinate hyperplane and let $\phi:\Gm^m \to \Gm^n$ be a torus
homomorphism given by an $n \times m$ integer matrix $A$. Define
\[ X:=\overline{\{\phi(t)y \mid y \in Y, t \in \Gm^m\}}, \]
so that $X$ is stable under the action of $\Gm^m$.  Let $X^0,Y^0$
be the open subsets of $X,Y$, respectively, where none of the
coordinates vanish. Then we have $\Trop(X^0)=\Trop(X) \cap \RR^n$ and
$\Trop(Y^0)=\Trop(Y) \cap \RR^n$ and 
\[ \Trop(X^0)=A\RR^m + \Trop(Y^0); \]
this follows, for instance, from \cite[Proposition 2.5]{Payne07}.
Let $\mu:\RR^m \times \Xan \to Z$ be the map constructed in
Section~\ref{sec:Torus}. We then obtain a continuous map
\[ \RR^m \times \Trop(Y) \to Z,\ (\tau,\eta) \mapsto
\mu(\tau,\sigma_Y(\eta)), \]
where $\sigma_Y$ is the section of $\Yan \to \Trop(Y)$ constructed in
Section~\ref{sec:Linear}. We would like to use this map to construct a
section $\Trop(X) \to Z$ of the surjection $Z \to \Trop(X)$, or at least
a section $\Trop(X^0) \to Z^0$, where $Z^0$ is the preimage of $X^0$ in
$Z^0$. There are two basic strategies for doing so. The first strategy
is given in the following proposition.

\begin{prop} \label{prop:SmearingA}
If the map $\RR^m \times \Trop(Y^0) \to \Trop(X^0),\ (\tau,\eta)\mapsto
A\tau + \eta$ has a continuous section, then so does the map $\trop:
Z^0 \to \Trop(X^0)$. Moreover, if the former section can be chosen
$\RR^m$-equivariant, then so can the latter.
\end{prop}

Here the action of $\RR^m$ on $\RR^m \times \Trop(Y^0)$ is given by
addition in the first coordinate and the trivial action on $\Trop(Y^0)$.

\begin{proof}
The composition 
\[ \sigma:(\Trop(X^0) \to \RR^m \times \Trop(Y^0) \to Z^0) \]
of a continous section $\Trop(X^0) \to \RR^m \times \Trop(Y^0)$ and
the map $(\tau,\eta) \mapsto \mu(\tau,\sigma_Y(\eta))$ is a section
$\Trop(X^0) \to Z^0$. The second statement is immediate.
\end{proof}

Here is an application of this construction. Recall from \cite{Develin05}
that the {\em tropical rank} of a real matrix is the largest size of
a square submatrix whose {\em tropical determinant} is attained by a
single term.

\begin{prop}
Let $m \leq p$ natural numbers. Let $X \subseteq \AA^{m \times p}$ be the
matrix variety defined by the vanishing of all $m \times m$-minors. On
$X$ acts $\Gm^m$ by scaling rows. Then the map $\trop:(X^0)^\an
\to \Trop(X^0)$ has a continuous section on the open subset $U$ of
$\Trop(X^0)$ consisting of matrices whose first $m-1$ columns form a
tropically non-singular matrix.
\end{prop}

\begin{proof}
Let $Y$ be the linear subspace contained in $X$ consisting of all matrices
$y$ such that $\one^T y = 0$, and let $\Gm^m$ act on matrices by scaling
rows. Let $A$ be the corresponding $(mp) \times m$-matrix of integers.
Then $X=\overline{\Gm^m Y}$ and hence $\Trop(X^0)=A \RR^{m} + \Trop(Y^0)$.
Now $\Trop(Y^0)$ is the set of matrices $\eta$ whose columns all
lie in the tropical hyperplane where the minimum of the coordinates is
attained at least twice. We will now argue that the map  
\[ \RR^m \times \Trop(Y^0 ) \to \Trop(X^0),\ (\tau,\eta) \to A\tau 
+ \eta \]
has an $\RR^m$-equivariant section over the open set $U$, defined as
follows. Let $\xi \in U$. Then for $\tau \in \RR^m$ the condition that
$\xi-A\tau$ lies in $\Trop(Y^0)$ is equivalent to the condition that for
each $j=1,\ldots,p$ the minimum $\min_i(\xi_{ij}-\tau_i)$ is attained
at least twice. This means that $-\tau$ lies on the intersection of
the $p$ tropical hyperplanes in $\RR^m$ with coefficient vectors given
by the columns of $\xi$. By the tropical non-singularity of the first
$m-1$ columns of $\xi$, the intersection of the corresponding $m-1$
hyperplanes is already spanned by a single vector, $-\tau$, unique
up to tropical scaling. 
For definiteness, choose $\tau_1$ equal to
$\xi_{11}$. We have that  $\tau$ depends continuously on $\xi$. 
Indeed the 
 {\em stable} intersection of $m-1$ hyperplanes, that in this case coincides 
with the intersection,
depends continuously on the hyperplanes (see
\cite[Section 5]{RichterGebert05} and \cite[Section 4]{Mikhalkin06a}).
Now $\xi-A \tau$ lies in
$\Trop(Y^0)$ and we can apply Proposition \ref{prop:SmearingA} to obtain a
$\RR^m$-equivariant section $U \to (X^0)^\an$. 
\end{proof}

\begin{re}
The map $U \to \RR^m,\ \xi \mapsto \tau$ constructed in the latter
proof can in general {\em not}
be extended to a continuous map $\Trop(X^0) \to \RR^m$ with the property
that $\xi-A\tau \in \Trop(Y^0)$ for all $\xi$. 
Indeed, consider the case where $m=p=4$, so that $X$ is the hypersurface
defined by a single determinant. Take two column vectors $a,b \in \RR^4$
in general position, so that the corresponding planes $H_a,H_b$ in tropical
projective $3$-space intersect in a tropical projective line of the form:
\begin{center}
\includegraphics[scale=.7]{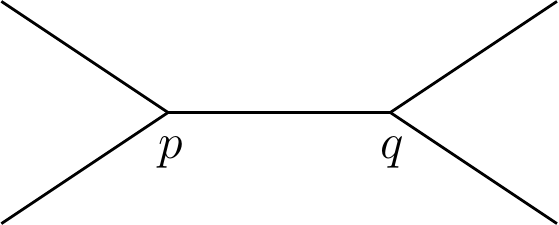}
\end{center}
Then the stable intersection of $H_a,H_a,H_b$ is one of the two trivalent
(projective) points, say $p$, and the stable intersection of $H_a,H_b,H_b$
is the other point, $q$. Now consider the matrix $\xi=(a|a|b|b)$.
Wiggling the first column slightly while keeping the remaining columns
fixed, the matrix stays within $\Trop(X^0)$ but now with the first three
columns defining hyperplanes that intersect in a single projective point
near $p$. Hence we see that for $\xi$ we need to take  $-\tau$ in
the stable intersection of $H_a,H_a,H_b$, i.e., in $p$, if we want it
to depend continuously on $\xi$. But wiggling the last column instead,
we find that we need to take $-\tau$ in $q$. Thus $-\tau$ cannot depend
continuously on $\xi$. 

We remark that the tropical multiplicity of such $\xi$ is typically
equal to two. After tropical scaling of rows and columns of $\xi$,
and after permuting rows if necessary, we have
\[ \xi=\begin{bmatrix}
0 & 0 & 0 & 0\\
0 & 0 & b_2 & b_2\\
0 & 0 & b_3 & b_3\\
0 & 0 & b_4 & b_4
\end{bmatrix} \]
where $0 \leq b_2 \leq b_3 \leq b_4$. The tropical determinant equals
$b_2$. Moreover, if $b_2 < b_3$,
then  $\ini_\xi 
\det(x)=(x_{13}x_{24}-x_{14}x_{23})(x_{31}x_{42}-x_{41}x_{32})$, which
defines a scheme with two irreducible components. In view
of \cite[Theorem 10.6]{Gubler14} it is conceivable that no continuous
section of $\trop$ near $\xi$ exists.
\end{re}

The second strategy for constructing a section $\Trop(X^0) \to (X^0)^\an$
is to show that the map $\RR^m \times \Trop(Y^0) \to (X^0)^\an$ factors
through the map $\RR^m \times \Trop(Y^0) \to \Trop(X^0)$. We will now
formulate sufficient conditions for this to happen.

The first of these conditions is purely polyhedral, namely, we require
that for each $\eta \in \Trop(Y^0)$ the set
\[ T_\eta:=\{\tau \in \RR^m \mid A\tau + \eta \in \Trop(Y^0)\}, \]
which is the support of a polyhedral complex, is connected. Observe that
these sets encode the ambiguity in the decomposition of $\xi$: if $\xi$
equals both $\eta_1+A\tau_1$ and $\eta_2+A\tau_2$, then $\tau_1-\tau_2
\in T_{\eta_1}$. Connectedness of $T_\eta$ means that there exists a
polyhedral path of decompositions of $\xi$ from the first
decomposition to the second. The second condition is more algebraic. Let $\eta \in
\Trop(Y^0)$. Extend the valuation $w:=\sigma_{Y}(\eta)$ from $K[Y]$
to the field $K(Y)$; this can be done since it sends no non-zero
polynomials to infinity. Let $y \in Y(K(Y))$ be the generic point of
$Y$. The coordinates of $y$ are thus $(x_1|_Y,\ldots,x_n|_Y)$, and the
vector of $w$-valuations of these coordinates is $\eta$.  By slight
abuse of notation, we write $w(y)=\eta$. Let $\tau \in \RR^m$ be such
that the line segment $[0,\tau]$ is contained in $T_\eta$. Then we
require that for all sufficiently small $\epsilon > 0$ there exists a
valued extension $(L,w_L)$ of $K(Y)$ and a $t \in \Gm^m(L)$ such that
$w_L(t)=\epsilon \tau$ (with the same abuse of notation: the vector of
valuations of the coordinates of $t$ equals $\epsilon \tau$) and $\phi(t)y\in Y^0(L)$.

\begin{prop} \label{prop:Smearing}
Suppose that the torus homomorphism $\phi:\Gm^m \to \Gm^n$ and the
linear space $Y$ satisfy the two aforementioned requirements. Then,
for $\xi=A\tau_1 + \eta_1 \in \Trop(X^0)$ with $\tau_1 \in \RR^m$ and
$\eta_1 \in \Trop(Y^0)$ the expression
\[ \sigma(\xi):=\mu(\tau_1,\sigma_Y(\eta_1)) \in Z \subseteq \Xan \]
does not depend on the chosen decomposition of $\xi \in \Trop(X^0)$.
The map $\sigma:\Trop(X^0) \to Z^0$ thus defined is a continuous,
$\RR^m$-equivariant section of the surjection $Z^0 \to \Trop(X^0)$.
\end{prop}

Before we give the proof, we discuss a simple example in the plane.

\begin{ex}
Let $Y \subseteq \AA^2$ be given by the linear equation $x_1-x_2=0$,
and let $\phi:\Gm \to \Gm^2$ be given by $\phi(t)=(t,t^{-1})$, so that
$A=(1,-1)^T$. Then we have
\[ \Trop(Y)=\{(\eta_1,\eta_2) \in \Rb^2 \mid \eta_1=\eta_2 \}
\text{ and } A\RR^1 = \{(\tau,-\tau) \mid \tau \in \RR\}. \]
We have $X=\overline{\phi(\Gm) Y}=\AA^2$ 
and $\Trop(X^0)=A\RR^1 + \Trop(Y^0)$ and
$T_\eta=\{0\}$ for all $\eta \in \Trop(Y^0)$. In the second requirement
we can just take $t=1$ for all $\eta$.  Thus both requirements are
met. For $\xi=(\xi_1,\xi_2)=A\tau + \eta=(\tau+\eta_1,-\tau+\eta_1)$
and $f=\sum_{i,j} c_{ij} x_1^i x_2^j$ we find
\[ \sigma(\xi)(f)=\min_{k \in \ZZ} 
\min_{i-j=k}(k\tau + v(c_{ij})+(i+j)\eta_1) 
=
\min_{i,j}(v(c_{ij})+i\xi_1+j\xi_2), 
\]
which extends to all of $\Trop(X)$, and in fact equals the section
obtained in Section~\ref{sec:Linear} when regarding $X$ as a linear space.
\hfill $\diamondsuit$
\end{ex}

\begin{proof}[Proof of Proposition~\ref{prop:Smearing}.]
For the first statement we need to prove that if $\xi$ can also be
decomposed as $A\tau_2 + \eta_2$ then
\[ \mu(\tau_2,\sigma_Y(\eta_2)) = 
\mu(\tau_1,\sigma_Y(\eta_1)). \]
This is equivalent to 
\[ \mu(\tau_2-\tau_1,\sigma_Y(\eta_2))=\mu(0,\sigma_Y(\eta_1)). \]
Now $\tau_1-\tau_2 \in T_{\eta_1}$, and since $T_{\eta_1}$ is connected,
by walking from $0$ to $\tau_1-\tau_2$ through $T_{\eta_1}$ along a
polyhedral path, it suffices to prove the following local version of this
equality. Let $\eta \in \Trop(Y^0)$ and $\tau \in \RR^m$ be such that the
segment $[0,\tau]$ lies entirely in $T_{\eta}$. Then we want to show that
\[ 
\mu(-\tau,\sigma_Y(A\tau+\eta))=\mu(0,\sigma_Y(\eta)).
\]
By definition of $\mu$, it suffices to prove this when applied to
a non-zero $f \in K[X]$ that is homogeneous with respect to the
$\Gm^m$-action, say of weight $\beta$. We will prove, in fact, that
the function
\[ \ell:[0,1] \to \RR,\ 
\epsilon \mapsto \mu(-\epsilon\tau,\sigma_Y(A\epsilon\tau+\eta))(f) \]
is constant on the interval $[0,1]$. Since $\ell$ is a continuous function
and $[0,1]$ is connected, it suffices to prove that $\ell$ has a local
minimum at every point in $[0,1]$. We give the argument at the point $0$;
it follows at other points in a similar manner.

Set $w:=\sigma_{Y}(\eta)$, and let $y \in Y(K(Y))$ be the generic
point. Then for $\epsilon>0$ sufficiently small a valued field extension
$(L,w_L) \supseteq (K(Y),w)$ and $t \in \Gm^m(L)$ exist as in the second
requirement, that is, with $w_L(t)= \epsilon\tau$ and $\phi(t)y \in
Y^0(L)$. After shrinking $\epsilon$ if necessary we may
assume that $\eta$ and $\epsilon A\tau + \eta$ are both compatible with
the same basis $J \subseteq [n]$ of the matroid defined by $Y$. Expand
the restriction $f|_{Y}$ as $\sum_{\alpha \in \NN^J} c_\alpha x^\alpha$.
Then on the one hand we have
\[ w_L(f|_{Y}(\phi(t)y))=w_L(t^\beta f|_{Y}(y))= 
\beta \cdot \epsilon \tau + \sigma_Y(\eta)(f), \]
where we have used that $\phi(t)y \in Y(L)$ and that $f$ is homogeneous of
$\Gm^m$-weight $\beta$. On the other hand, we have 
\[ w_L(f|_{Y}(ty))=w_L(\sum_{\alpha \in \NN^J} c_\alpha t^{\alpha A} y^\alpha)
\leq \min_\alpha (v(c_\alpha)+\alpha \cdot A \epsilon \tau+\alpha \cdot \eta) 
=\sigma_Y(A\epsilon \tau+\eta)(f). \]
Thus we find that
\[ \ell(\epsilon)=\sigma_Y(A\epsilon\tau+\eta)(f) - \beta\cdot \epsilon \tau 
\geq \sigma_Y(\eta)(f)=\ell(0), \]
as desired. This shows that the section $\sigma:\Trop(X^0) \to
Z$ is well-defined. To see that $\sigma$ is continuous, decompose
$\Trop(Y^0)$ into finitely many closed polyhedra $P_i$ and let $P_i'$
denote the image of 
\[ \RR^m \times P_i \to \Trop(X^0),\ (\tau,\eta)
\mapsto A\tau+\eta. \] 
By basic linear algebra over $\RR$, on each $P_i'$ this map has a
continuous (in fact, affine-linear) section $P_i' \to \RR^m \times
P_i$. This shows that the restriction of $\sigma$ to each $P_i'$
is continuous. Since the $P_i'$ form a finite closed cover of
$\Trop(X^0)$, the map $\sigma$ is continuous on $\Trop(X^0)$. 

Finally, we need to verify that $\sigma$ is $\RR^m$-equivariant.
Let $\xi=A\tau+\eta \in \Trop(X^0)$ with $\tau \in \RR^m$ and $\eta \in
\Trop(Y^0)$. Let $\tau' \in \RR^m$. Then we have
\[ 
\sigma(A\tau'+\xi)=
\sigma(A(\tau'+\tau)+\eta)=
\mu(\tau'+\tau,\sigma_Y(\eta))=
\mu(\tau',\mu(\tau,\sigma_Y(\eta)))=
\mu(\tau',\sigma(\xi)),
\]
as desired.
\end{proof}

\begin{re}
While Propositions~\ref{prop:SmearingA} and~\ref{prop:Smearing} give
sections only over $\Trop(X^0)$, we will see that, at least in the
cases of Grassmannians of planes and of the variety of rank-two matrices,
sections exists over all of $\Trop(X)$.
\end{re}

\section{Grassmannians of planes} \label{sec:Grassmannian}

In this section we set $n:=\binom{m}{2}$ and consider $\AA^n$
with coordinates $x_{ij}$ for $1 \leq i < j \leq n$. We also write
$x_{ji}=-x_{ij}$ for $i>j$, and $\xi_{ji}=\xi_{ij}$ for tropical
coordinates. Let $X:=\cGr(2,m) \subseteq \AA^n$ denote the affine cone
over the Grassmannian of planes, given as the image of the polynomial map
\[ \psi:(\AA^{m})^2 \to \AA^n, 
(y,z) \mapsto (y_i z_j - y_j z_i)_{i<j}. \]
This map is $\Gm^m$-equivariant with respect to the standard (diagonal)
action of $\Gm^m$ on $(\AA^m)^2$ and the action of $\Gm^m$ on $\AA^n$ via
$\phi:\Gm^m \to \Gm^n$ given by $\phi(t):=(t_i t_j)_{i<j}$.  The dense
subset of $(\AA^m)^2$ where $z$ has no non-zero entries equals $\Gm^m
\cdot (\AA^m \times \{\one\})$. Consequently, if we
set 
\[ Y:=\psi(\AA^m \times \{\one\}) \subseteq X, \]
where $\one$ is the all-one vector, then we have $X=\overline{\phi(\Gm^m)
\cdot Y}$. Note that $Y$ is a linear space, with generic
point $(y_i-y_j)_{i<j}$; hence we are in the setting of
Section~\ref{sec:Smearing}. Let $\mu:\RR^m \times \Xan \to Z \subseteq
\Xan$ be the map from Section~\ref{sec:Torus}, which restricts to an
action of $\RR^m$ on $Z$, and let $A \in \ZZ^{n \times m}$ be the matrix
corresponding to $\phi$. We will prove the following theorem.

\begin{thm} \label{thm:Grassmannian}
The surjective projection from $Z \subseteq \cGr(2,m)^\an$ to
$\Trop(\cGr(2,m))$ has a continuous, $\RR^m$-equivariant section.
\end{thm}

A version of this theorem first appeared in \cite{Cueto14}; see
Remark~\ref{re:relationtoCueto} below. Lifts from $\Trop(\cGr(2,m))$
into tropicalisations of other flag varieties were constructed in
\cite{Iriarte10,Manon12}.

Our proof consists of two parts. We first construct a continuous section
in the spirit of Proposition~\ref{prop:SmearingA}, which relies on the
choice of a hyperplane in $\Trop(\PP^{m-1})$. Then we use the technique
of Proposition~\ref{prop:Smearing} to verify that the constructed section
is, in fact, natural and independent of the choice of hyperplane. This
then also implies $\RR^m$-equivariance.

We will use that the matroid on the variables $x_{ij}$ defined by $Y$
is the graphical matroid of the complete graph $K_m$. This is immediate
from the definition of $Y$, and was also exploited in \cite[Section
4]{Ardila06}. Thus a basis $J$ as in Section~\ref{sec:Linear} is a tree
with vertex set $[m]$. We will write $\Gamma$ instead of $J$. Given
such a tree $\Gamma$, one finds all $\eta \in \Trop(Y)$ compatible
with $\Gamma$ as follows (see also Remark~\ref{re:constant}). First,
give arbitrary values in $\Rb$ to all $\eta_{ij}$ with $ij$ an edge in
the tree $\Gamma$. Then, for each ege $ij$ in $K_m \setminus \Gamma$ set
$\eta_{ij}$ equal to the minimum of the $\eta_{kl}$ over all edges $kl$
in the simple path from $i$ to $j$ in $\Gamma$. See Figure \ref{tree}.

\begin{figure}
\begin{center}
\includegraphics[width=.7\textwidth]{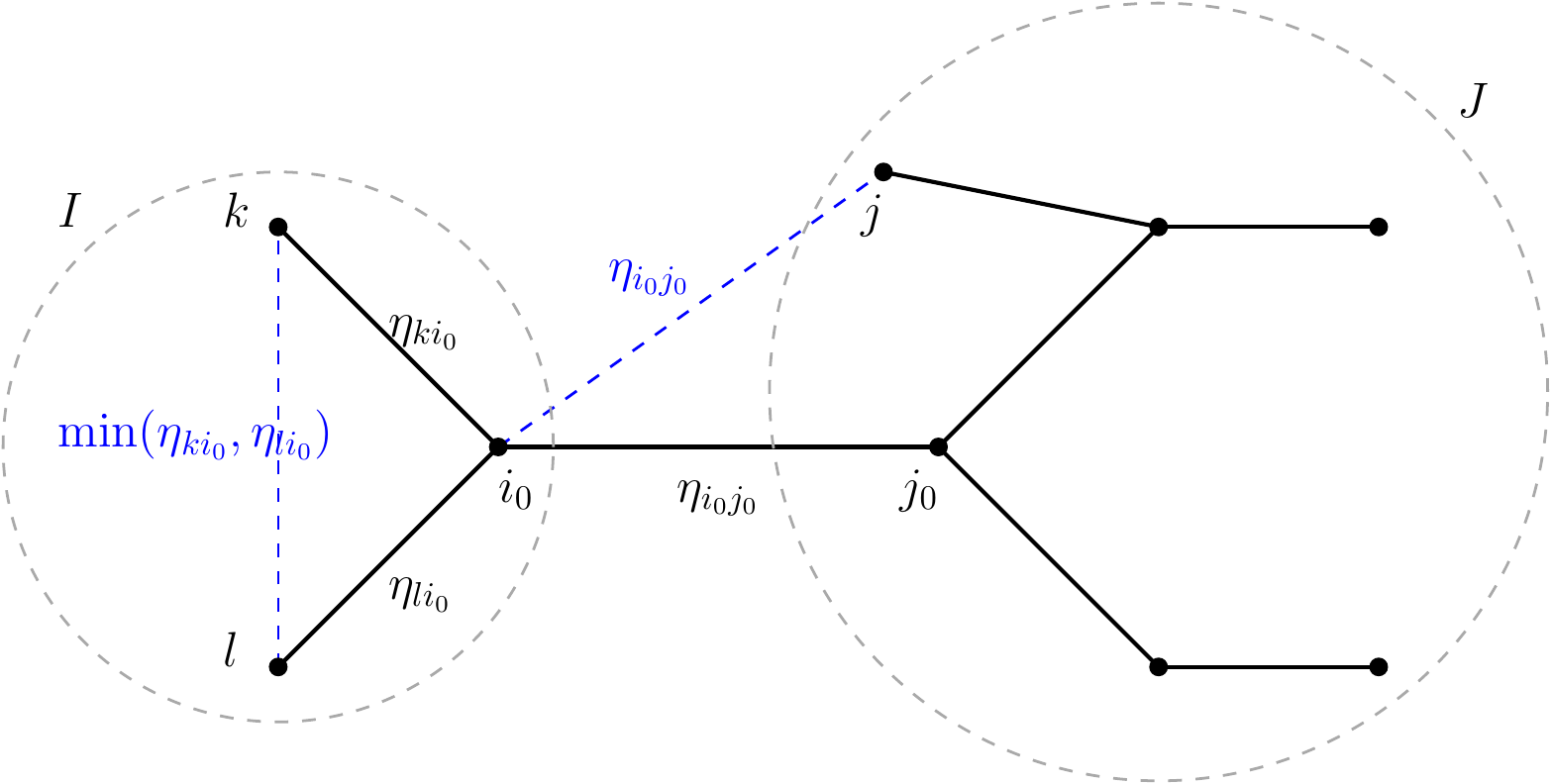}
\end{center}
\caption{A spanning tree in $K_{9}$ with minimal-weight edge $i_0j_0$}\label{tree}
\end{figure}

Up to tropical scaling, the points of $\Trop(X)\setminus\{\infty\}$
are in one-to-one correspondence with tropical projective lines in the
simplex $\Delta:=\Trop(\PP^{m-1})$ (see \cite[Theorem 3.8]{Speyer04}
for $\Trop(X^0)$). Under this correspondence the point $(\xi_{ij})_{i<j}$
gives rise to the tropical projective line consisting of points $\zeta$
for which $\min\{\xi_{ij}+\zeta_k,\xi_{ik}+\zeta_j,\xi_{jk}+\zeta_i\}$
is attained at least twice for each $1 \leq i<j<k \leq m$. We will use
the following characterisation of $\Trop(Y) \subseteq \Trop(X)$.

\begin{lm}
Under the correspondence above, the points of $\Trop(Y)$ correspond
bijectively to the tropical lines that pass through the all-zero
point $0$.
\end{lm}

\begin{proof}
First, for $\eta \in \Trop(Y)$, choose a lift $(x_{ij}:=y_i-y_j)_{i<j}$
in $Y$ with $v(x_{ij})=\eta_{ij}$. Then the $3 \times 3$-subdeterminant
of the matrix $(y|\one|\one)^T$ in columns $i<j<k$ equals
$0=x_{ij}-x_{ik}+x_{jk}$. Hence $\min\{\xi_{ij},\xi_{ik},\xi_{jk}\}$
is attained at least twice, i.e., $0$ lies on the tropical line
corresponding to $\eta$.

Conversely, suppose that $0$ lies on the tropical line
corresponding to $\eta$, i.e., that for all $i<j<k$ the minimum
$\min\{\eta_{ij},\eta_{ik},\eta_{jk}\}$ is attained at least twice. Equip
$K_m$ with edge weights given by $\eta$. Then in each triangle $\{i,j,k\}$
the minimum edge weight is attained at least twice. An easy induction
then shows that in each {\em cycle} the minimum edge weight is also
attained at least twice (see also Figure~\ref{fig:cycle} for a similar
argument for the graphical matroid of the complete bipartite graph,
where triangles are replace by four-cycles).  Since these cycles are
precisely the circuits of the matroid of $Y$, which form a tropical
basis by \cite{Ardila06,Bogart07}, $\eta$ lies in $\Trop(Y)$.
\end{proof}

Note that the action of $\RR^m$ on $\Trop(X)$ is by translation of the
tropical lines.

\begin{proof}[Proof of Theorem~\ref{thm:Grassmannian}, construction of a continuous section.]

To go from $\xi \in \Trop(X)$ to a pair $(\tau,\eta) \in \RR^m
\times \Trop(Y)$ one is tempted to proceed as follows. Let
$\ell$ be the line represented by $\xi$, let $\tau \in \RR^m$ be
such that $-\tau+\RR(1,\ldots,1)$ is a point on $\ell$, and set
$\eta:=-A\tau+\xi$. Then $\eta$ represents the translate of $\ell$
by $-\tau$, which therefore passes through $0$.  By construction, the
pair $(\tau,\eta)$ satisfies $A\tau + \eta=\xi$, so that the valuation
$\sigma(\xi):=\mu(\tau,\sigma_Y(\eta))$ maps to $\xi$.

There are various problems with this definition of $\sigma$, but
we can sharpen it as follows. A first, minor problem is that if
$\xi=\infty(=(\infty,\ldots,\infty))$, then $\xi$ does not represent
a line.  In that case, we just set $\sigma(\xi)$ equal to $\infty \in
\Xan$. The second, and more serious, problem is that $\ell$ may not
contain points $\tau \in \RR^m/\RR(1,\ldots,1)$. To remedy this, we will
use a stratification of $X=\cGr(2,m)$ and $\Trop(X)$ defined as follows
(and also used, in slightly different terminology, in \cite{Cueto14}).
 For $x \in X$ let $J_x \subseteq [m]$ be the set of $i$ for which there
exists a $j \neq i$ with $x_{ij}\neq 0$.    
Note that $J_x$ is either empty,
or else has cardinality at least two. For any subset $J \subseteq [m]$
of cardinality zero or at least two we define
\[ X_J:=\{x \in X \mid J_x=J\}. \] 
This stratum is a locally closed subset of $X$, and $X$ is the disjoint
union of these strata. The stratum $X_\emptyset$ consists of $0$ only,
while for $|J| \geq 2$ the stratum $X_J$ is the $\psi$-image of the subset
of $(\AA^m)^2$ where $y,z$ are linearly independent and $(y_i,z_i)=(0,0)$
if and only if $i \not \in J$. So $X_{[m]}$ is the largest one among
these strata, and it parameterises lines that intersect $\Gm^m$. Each
$X_J$ is the dense set in the (cone over the) smaller Grassmannian of
$2$-spaces contained in $\AA^J \times \{0\}^I$ consisting of all spaces
that intersect $\Gm^J \times \{0\}^I$. Similarly, points of $\Trop(X_J)$
parameterise tropical lines in the face $\Delta_J$ of $\Delta$ (where
all $I$-coordinates are $\infty$) that intersect the relative interior
of $\Delta_J$.  Let $Y_J$ denote the $J$-analogue of $Y$, that is, the
subspace of $K^{\binom{J}{2}}$ parameterised by $(y_j-y_{j'})_{j<j'}$,
identified with a subspace of $K^{\binom{m}{2}}$ by extending with
zero coordinates.  Then $Y_J\setminus \{0\}$ is a subset of $X_J$,
and in fact we have $\phi(\Gm^J) \cdot (Y_J \setminus \{0\})=X_J$. Note
also that $Y_J$ is not a {\em subspace} of $Y=Y_{[m]}$ but rather its
{\em image} under projecting some coordinates to $0$. Note that both
$Y_\emptyset=X_\emptyset=\{0\}$.

We choose $\tau \in \Rb^m$ as a function of $\xi$ as follows. If
$\xi=\infty$, then set $\tau:=\infty$.  Otherwise, let $H$ be the tropical
hyperplane in $\Delta$ with the tropical equation $\zeta_1 \oplus \cdots
\oplus \zeta_m$, and let $\tau$ represent the stable intersection of
$H$ and the tropical line $\ell$ represented by $\xi$. By continuity of stable
intersection, the projective point $\tau + \RR(1,\ldots,1)$ depends
continuously on non-infinite $\xi$.

Next, we choose $\eta$ as a function of $\xi$. If $\xi=\infty$, then set
$\eta:=\infty$. Otherwise, let $J$ of cardinality at least two be such
that $\xi \in \Trop(X_J)$. Then $\ell$ lies in $\Delta_J$ and intersects
the relative interior of $\Delta_J$. As a consequence, $\tau_j \neq
\infty$ if and only if $j \in J$. Set $\eta_{ij}:=\xi_{ij}-\tau_i-\tau_j$
for $i,j \in J$ and $\eta_{ij}:=\infty$ if one of $i,j$ lies in $I$. Then
$\eta$ lies in $\Trop(Y_J)$ (and this also holds for $\xi=\infty$,
in which case $J=\emptyset$).

The pair $(\tau,\eta)$ thus constructed does not depend continuously on
$\xi$, but we claim that the valuation
\[ \sigma(\xi):=\mu(\tau,\sigma_{Y_J}(\eta)) \in \Xan \]
does. Here we abuse notation slightly, since $\tau$ will in general have
some coordinates equal to $\infty$---but one readily verifies that,
since $A$ contains only non-negative entries, $\mu$ extends to $\Rb^m
\times \Xan$. By construction, we have $A\tau+\eta=\xi$, and this implies
that $\sigma(\xi) \in \Xan$ does indeed map to $\xi$.

First observe that tropically scaling all coordinates of
$\tau$ with $c \in \RR$ and all coordinates of $\eta$ with $-2c$ leads
to the same valuation. Now let $\xi^{(p)},p=1,2,3,\ldots$ be a sequence of
points in $\Trop(X)$ that converges to a non-infinity limit $\xi \in
\Trop(X_J)$ with $|J| \geq 2$. After deleting an
initial segment of the sequence, we may assume that each $\xi^{(p)}$
lies in some $\Trop(X_{J^{(p)}})$ with $J^{(p)} \supseteq J$. Let
$\eta^{(p)} \in \Trop(Y_{J^{(p)}})$ and $\tau^{(p)} \in \RR^m$ be
the corresponding points, so that $\xi^{(p)}= A\tau^{(p)}+\eta^{(p)}$
for all $p$. The projective points $\tau^{(p)}
+\RR(1,\ldots,1)$ converge to $\tau+\RR(1,\ldots,1)$
(by continuity of stable intersection). Hence, after suitable tropical scalings
of the $\tau^{(p)}$ and the $\eta^{(p)}$, we achieve that $\tau^{(p)}
\to \tau$ for $p \to \infty$. Then for $i,j \in J$ we find that
\[ \eta_{ij}^{(p)}=\xi^{(p)}_{ij}-\tau^{(p)}_i-\tau^{(p)}_j \to
\xi_{ij}-\tau_i-\tau_j = \eta_{ij} \text{ for } p \to
\infty. \]
We now argue that for each $\Gm^m$-homogeneous element $f \in K[X]$
the value $\sigma(\xi^{(p)})(f)$ converges to $\sigma(\xi)(f)$.
Let $\beta \in \NN^m$ be the weight of $f$. If $\beta_i>0$ for some $i
\not \in J$, then $f$ lies in the ideal generated by the coordinates
$x_{kj}$ for which one of $k,j$ does not lie in $J$. In this case,
$\sigma(\xi)(f)=\infty$. To see that $\sigma(\xi^{(p)})(f)$ tends to
infinity, expand $f=\sum_{ij} x_{ij} f_{ij}$ where the sum
is over pairs $(i,j)$ that are not both in $J$. Then we have
\[ \sigma(\xi^{(p)})(f) \geq \min_{ij}(\xi^{(p)}_{ij} +
\sigma(\xi^{(p)})(f_{ij})). \]
Since each $\xi^{(p)}_{ij}$ tends to infinity and each
$\sigma(\xi^{(p)})(f_{ij})$ is bounded from below, we find the desired
convergence (a similar convergence argument applies when the limit $\xi$
equals $\infty$). If $\beta_i=0$ for all $i \not \in J$, then $f$ depends
only on the coordinates $x_{ij}$ with $i,j \in J$, and it suffices to
show that
\[ \sigma_{Y_{J^{(p)}}}(\eta^{(p)})(f) \to \sigma_{Y_J}(\eta)(f),
\quad p \to \infty. \]
Using the definition of $\sigma$ and the fact that $\eta^{(p)}_{ij}
\to \eta_{ij}$ for $p \to \infty$ and $i,j \in J$, this convergence
follows if there exists a tree on $J^{(p)}$ compatible with $\eta^{(p)}$
which contains a spanning tree on $J \subseteq J^{(p)}$. But this is a
consequence of the basis exchange axiom: start with any tree $\Gamma$ on
$J^{(p)}$ compatible with $\eta^{(p)}$. If the induced forest $\Gamma|_J$
on $J$ is not connected, pick arbitrary endpoints $j,j' \in J$ that
belong to different connected components of $\Gamma|_J$. Then replace,
in $\Gamma$, an edge in the simple path from $j$ to $j'$ of smallest
$\eta^{(p)}$-weight by $jj'$ (which has the same weight).  This creates
a new $\Gamma$ compatible with $\eta^{(p)}$ such that $\Gamma|_J$ has
fewer connected components than before. Proceed in this fashion until
$\Gamma|_J$ is connected. See Figure \ref{fig:forest} for an illustration 
of this procedure. This concludes the proof that $\sigma$ is a
continuous section $\Trop(X) \to Z$ of the surjection $Z \to \Trop(X)$.
\end{proof}

\begin{figure}
\begin{center}
\includegraphics[width=.9\textwidth]{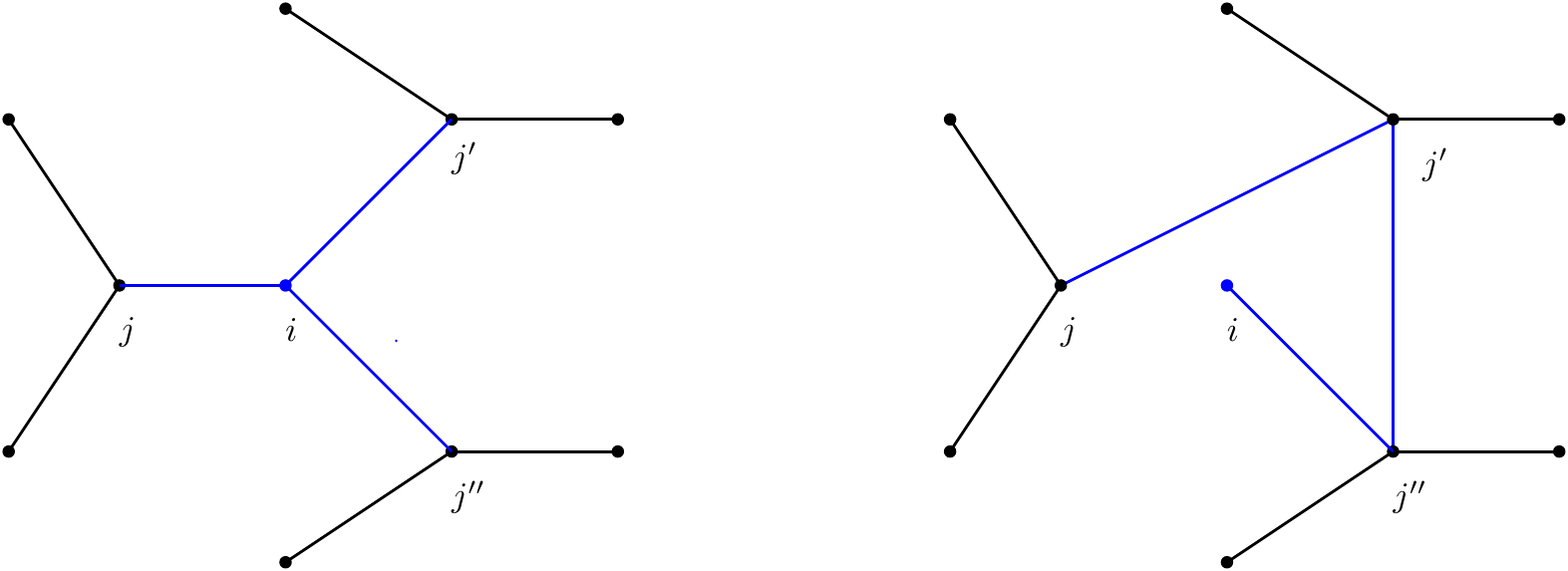}
\end{center}
\caption{Basis exchange in the case $ J^{(p)}\setminus J=\{i\}$, with
 $\eta^{(p)}_{ij}\le\eta^{(p)}_{ij'}\le \eta^{(p)}_{ij''}$.}\label{fig:forest}
\end{figure}

\begin{proof}[Proof of Theorem~\ref{thm:Grassmannian},
naturality and equivariance.]

In the previous proof, we decomposed $\xi$ as $A\tau + \eta$ by choosing
for $-\tau$ a point on the tropical line $\ell$ represented by $\xi$. This
point was obtained by stably intersecting $\ell$ with a hypersurface. By
verifying the conditions of Proposition~\ref{prop:Smearing}, we now
show that the chosen decomposition is, in fact, irrelevant for the
section $\sigma$.

First, if $\eta \in \Trop(Y^0)$ corresponds to a tropical projective
line $\ell$, then the $\tau \in \RR^m$ for which $A\tau + \eta$
lies in $\Trop(Y)$ are those for which $-\tau$ lies in $\ell$. Thus
$T_\eta=-\ell$ is connected. This settles the first requirement for
Proposition~\ref{prop:Smearing}.

Now let $\eta \in \Trop(Y^0)$ and $\tau \in \RR^m \setminus
\{0\}$ such that $[0,\tau]$
lies in $T_\eta$. For sufficiently small $\epsilon > 0$ there exists
a tree $\Gamma$ compatible with both $\eta$ and $\eta':=A\epsilon
\tau+\eta$. Indeed, the points in $\Trop(Y^0)$ compatible with any given
tree form (the support of) a closed, finite polyhedral complex. There
are finitely many trees, and they give finitely many polyhedral complexes that
together cover $\Trop(Y^0)$. But then $[\eta,A\tau+\eta]$ has an initial
segment entirely contained in one of these complexes. By a similar
argument, by shrinking $\epsilon$, we may moreover assume that there
exists an edge $i_0j_0$ in $\Gamma$ such that both
\[ \eta_{i_0j_0} \leq \eta_{ij} \text{ and }
\eta'_{i_0j_0} \leq \eta'_{ij}
\text{ for all } ij \in \Gamma \]
(and hence also for all $ij \in K_m \setminus \Gamma$). The edge
$i_0j_0$ cuts the tree $\Gamma$ into two connected components (see
Figure~\ref{tree}). Let $[m]=I \cup J$ be the vertex sets of these
connected components, with $i_0 \in I$ and $j_0 \in J$.  We claim
that $\tau_i=\tau_{i_0}$ for all $i \in I$ and $\tau_j=\tau_{j_0}$
for all $j \in J$. Indeed, pick $j \in J$ and consider the cycle
formed by $i_0,j,$ and then back along $\Gamma$ to $i_0$. We have
$\eta'_{i_0j}=\eta'_{i_0j_0}$, since this is the edge of $\Gamma$ in
said cycle with smallest $\eta'$-weight. On the other hand, we have
\[ \eta'_{i_0j}=\eta_{i_0j} + \epsilon(\tau_{i_0} +
\tau_{j}) = \eta_{i_0j_0} + \epsilon(\tau_{i_0} + \tau_{j})
\]
and 
\[ \eta'_{i_0j_0}=\eta_{i_0j_0} + \epsilon(\tau_{i_0} +
\tau_{j_0}). \]
This shows that $\tau_j=\tau_{j_0}$. Similarly, we find that for all $i
\in I$ we have $\tau_i=\tau_{i_0}$.

To construct $t$, we may assume that one of $\tau_{i_0},\tau_{j_0}$
is zero and the other positive---indeed, this can be achieved by
adding a multiple of the all-one vector to $\tau$, which can be mimicked
by multiplying $t$ with a scalar of the right valuation.  Without loss
of generality, suppose that $\epsilon\tau_{i_0}=:a$ is positive
and $\tau_{j_0}$ is zero. Then adding $\epsilon A\tau$ to $\eta$ has
the effect of increasing all $\eta_{ij}$ with $i,j \in I$ by $2a$,
keeping all $\eta_{ij}$ with $i,j \in J$ constant, and increasing all
$\eta_{ij}$ with $i \in I$ and $j \in J$ by $a$. As, by assumption,
the minimal-weight edge in $\Gamma$ remains the edge $i_0j_0$, the
minimal $\eta$-weight of an edge of $\Gamma$ with both vertices in $J$
must be at least $\eta_{i_0j_0}+a$.

Now let $w=\sigma_Y(\eta)$ as in the proof of
Theorem~\ref{thm:Linear} and let $y \in Y^0$ be the generic point. Its
coordinates are $x_{ij}=(y_i-y_j)_{ij}$ where the $y_i$ are variables. It
represents the subspace spanned by the rows of the matrix
\[ 
\begin{bmatrix}
y_1 & y_2 & \cdots & y_m\\
1   &  1  & \cdots &  1
\end{bmatrix}.
\]
A point $t$ sends $y$ into $Y^0$ if and only if it sends this subspace
to another subspace containing the all-one vector, hence if and only if
$t^{-1}$ lies in the subspace. Hence we make the {\em Ansatz}
\[ t_i=\frac{1}{c(y_i-y_{j_0})+d}, \]
where $c,d$ will be chosen from $K$. Then by definition of $w$, we have
\[ w(c(y_i-y_{j_0})+d)=\min\{v(c)+w(y_i-y_{j_0}),v(d)\}. \]
Now for each $i$ the expression
$y_i-y_{j_0}$ expands as a sum of the $x$-variables corresponding to the
edges on the path in $\Gamma$ from $i$ to $j_0$, and $w(y_i-y_{j_0})$
equals the minimal $\eta$-weight among these edges. For $i \in I$ this
equals $\eta_{i_0j_0}$, as this is the minimal-weight edge overall. Thus
we choose $c \in K$ such that $v(c)+\eta_{i_0j_0}=-a<0$. For $i \in J$
the minimal weight along the path is at least $a+\eta_{i_0j_0}$, so if
we choose a $d \in K$ with valuation $0$, then the denominator in the
{\em Ansatz} gets the right valuation for both $i \in I$ and $i \in J$.

We have thus constructed a $t \in \Gm(K(Y))$ with
$w(t)=\epsilon\tau$ and $\phi(t) y \in Y(K(Y))$. As all requirements
of Proposition~\ref{prop:Smearing} are met, we have constructed an
$\RR^m$-equivariant section $\Trop(X^0) \to Z^0$. This section agrees with
the restriction to $\Trop(X^0)$ of the section constructed in the previous
proof. Hence the continuous section constructed there is $\RR^m$-equivariant.
\end{proof}

\begin{re} \label{re:relationtoCueto}
In \cite{Cueto14} the setting is projective rather than affine. Theorem
1.1 and Corollary 7.3 from that paper follow from our theorem by applying
Lemmas~\ref{lm:projective} and~\ref{lm:quotient}, respectively.
\end{re}

\section{Rank-two matrices} \label{sec:RankTwo}

In this section we take $n=m\cdot p$ and consider $\AA^n$ with
coordinates $x_{ij}$ with $1 \leq i \leq m$ and $1 \leq j \leq p$.
Let $X\subset \A^{n}$ be the image of the polynomial map
$$
\psi: (\A^m)^2\times(\A^p)^2\to \A^{n}, \quad
 (y,y'),(z,z')\mapsto(y_iz'_j-y'_iz_j)_{1\le i\le m, 1\le
 j\le p}.
$$
It is the variety of matrices of rank at most two, and also the
affine cone over the variety of secant lines of the Segre embedding
of $\PP^{m-1}\times\PP^{p-1}$ in $\PP^{n-1}$. It is an irreducible
determinantal variety of dimension $2(m+p-2)$. 

Let $Y$ be the subvariety of $X$ defined as the image of 
$$
(\A^m\times\{\mathbf{1}\})\times(\A^p\times\{\mathbf{1}\})
$$
via $\psi$. Points of $Y$ have coordinates
$x_{ij}=(y_i-z_j)$ in $\A^{n}$, so that $Y$ is the zero locus of the
linear forms
\begin{equation} \label{eq:forms}
x_{ij}+x_{lk}=x_{ik}+x_{lj},\ 1\le i<l\le m, 1\le j<k\le p.
\end{equation}
Consider also the homomorphism of tori given by 
$$
\phi: \G_m^m\times\G_m^p\to\G_m^{n}, (t,s)\mapsto
(t_is_j)_{1\le i\le m, 1\le j\le p}.
$$
The corresponding $n\times(m+p)$-matrix $A$ has a one-dimensional
kernel spanned by $(1,\ldots,1,-1,\ldots,-1)$. We have
$X=\overline{\phi(\Gm^m\times\Gm^p)\cdot Y}$ (this can be proved using
equivariance of $\psi$, as in Section~\ref{sec:Grassmannian}), and
$$
\Trop(X^0)=A(\R^m \times \R^p)+\Trop(Y^0)
$$
where $X^0 \subseteq X$ and $Y^0 \subseteq Y$ are the loci where no
coordinate is zero. Let $\mu:\RR^{m+p} \times \Xan \to Z \subseteq \Xan$
be as constructed in Section~\ref{sec:Torus}. We will prove the following
theorem.

\begin{thm} \label{thm:RankTwo}
The surjection $\Xan \to \Trop(X)$, where $X$ is the variety of $m \times
p$-matrices of rank at most two, has a continuous, $\RR^m$-equivariant
section $\Trop(X) \to Z$.
\end{thm}

Note that we do not claim that the section is also
$\RR^p$-equivariant. While this might be the case, our construction
below does not yield this.

\begin{figure}
\includegraphics{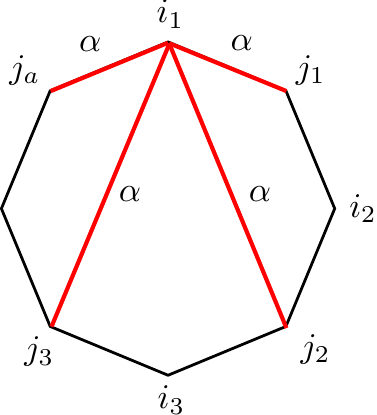}
\caption{Only four-cycles need to be tested for membership
of $\Trop(Y^0)$.} \label{fig:cycle}
\end{figure}

For the proof of this theorem, we need to understand the points in
$\Trop(X)$ and its tropical subvariety $\Trop(Y)$. By \cite[Corollary
3.8]{Develin05}, a matrix $\xi \in \RR^{m \times p}$ lies in $\Trop(X^0)$
if and only if it has tropical rank at most $2$, i.e., if and only if
all its $3 \times 3$-submatrices are tropically singular. This extends
directly to all of $\Trop(X)$.  To understand $\Trop(Y)$ note that
the matroid defined by $Y$ is the graphical matroid of the complete
bipartite graph $K_{m,p}$; this is immediate from the parameterisation
$x_{ij}=y_i-z_j$. In other words, $\eta \in \Rb^{m \times p}$ lies in
$\Trop(Y)$ if and only if along each cycle in $K_{m,p}$ the minimal
$\eta$-weight of an edge is attained at least twice.  We claim that this
is equivalent to the condition that in every $2 \times 2$-submatrix of
$\eta$ the minimal entry appears at least twice. Indeed, necessity of
the latter condition is obvious, as any $2 \times 2$-submatrix records
the weights of a $4$-cycle in $K_{m,p}$.  For sufficiency, assume that
the minimal $\eta$-weight in every $4$-cycle is attained at least twice,
and let $C$ be a general (simple, even) cycle in $K_{m,p}$. Label $C$
as $i_1-j_1-i_2-j_2-\cdots-i_a-j_a-i_1$, where the $i$s are in $[m]$ and
the $j$s are in $[n]$ and where $\alpha:=\eta_{i_1,j_1}$ is the minimal
weight of an edge in $C$.  Assume, for a contradiction, that all other
edges in $C$ have $\eta$-weight strictly larger than $\alpha$. Then
in the $4$-cycle $i_1-j_1-i_2-j_2-i_1$ the weight $\eta_{i_1,j_2}$
must be $\alpha$. Next, in the $4$-cycle $i_1-j_2-i_3-j_3-i_1$
the weight $\eta_{i_1,j_3}$ must also equal $\alpha$, etc. In this
manner we find that $\eta_{i_1,j_a}$ must also equal $\alpha$, a
contradiction. See Figure~\ref{fig:cycle} for an illustration. Armed
with this characterisation of $\Trop(Y)$ we will now prove the theorem.

\begin{proof}[Proof of Theorem~\ref{thm:RankTwo}.]
As in the proof of Theorem~\ref{thm:Grassmannian}, we use a stratification
of $X$. For $I \subseteq [m]$ and $J \subseteq [p]$ let $X_{IJ}$ denote
the locus in $X$ consisting of $x$ such that the rows of $x$ labelled
by $[m]\setminus I$ and the columns of $x$ labelled by $[p] \setminus J$
are identically zero and the submatrix $x[I,J]$ does not have identically
zero rows or columns. Let $Y_{IJ}$ denote the $(I,J)$-analogue of $Y$. It
is the image of $Y$ under the map sending all coordinates outside the
$[I,J]$-submatrix to zero.

For $\xi \in \Trop(X_{IJ})$ we let $\tau \in \Rb^m$ be the tropical
product $\xi \odot (0,\ldots,0)^T$, a point in the tropical convex hull
of the columns of $\xi$. Then we have $\tau_i \neq \infty$ if and only
if $i \in I$. Let $\xi' \in \Trop(X_{IJ})$ be the matrix obtained
from $\xi$ by subtracting $\tau_i$ from each $\xi_{ij}$ with $i \in I,
j \in J$. Then let $\rho \in \Rb^p$ be the tropical product $(0,\ldots,0)
\odot \xi'$, which records the minimal entry in each column of $\xi'$.
Let $\eta$ be the matrix obtained from $\xi'$ by subtracting $\rho_j$ from
each $\xi'_{ij}$ with $i \in I, j \in J$. By \cite[Lemma 6.2]{Develin05},
the matrix $\eta[I,J]$ has the property that in each of its $2 \times
2$-submatrices the minimal entry appears at least twice. By the discussion
preceding the proof, $\eta$ lies in $\Trop(Y_{IJ})$.

We set
\[ \sigma(\xi):=\mu((\tau,\rho),\sigma_{Y_{IJ}}(\eta)), \]
and claim that this depends continuously on $\xi$. To see this, let
$\xi^{(q)},\ q=1,2,\ldots$ be a sequence in $\Trop(X)$ converging to
$\xi \in \Trop(X_{IJ})$, and construct $\tau^{(q)}$ and $\rho^{(q)}$
and $\eta^{(q)}$ as above. After dropping finitely many initial terms,
we have $\xi^{(q)} \in \Trop(X_{I^{(q)}J^{(q)}})$ with $I^{(q)} \supseteq
I$ and $J^{(q)} \supseteq J$. For $i \in I$ and $j \in J$ we find that
$\tau^{(q)}_i \to \tau_i$ for $q \to \infty$ and also $\lim_{q \to
\infty} \rho^{(q)}_j = \rho_i$ and $\lim_{q \to \infty} \eta^{(q)}_{ij}
= \eta_{ij}$. We will not need the limits of the remaining
entries of $\rho^{(q)}, \eta^{(q)}, \eta^{(q)}$. 

Let $f$ be a $\Gm^{m+p}$-weight (i.e., multi-homogeneous) element of
$K[X]$. We have the same dichotomy as in the proof for the Grassmannian
case: either $f$ lies in the ideal generated by all variables $x_{ij}$
with $i \not \in I$ or $j \not \in J$, and in this case
\[ \sigma(\xi^{(q)})(f) \to \infty = \sigma(\xi)(f) \text{
for } q \to \infty; \]
or $f$ lies in the ring generated by the $x_{ij}$ with $i,j \in J$. In
the latter case, it suffices to show that
\[ \sigma_{Y_{I^{(q)}J^{(q)}}}(\eta^{(q)})(f) \to 
\sigma_{Y_{IJ}}(\eta)(f). \]
Proceeding as for the Grassmannian of $2$-spaces, we find that there exists,
for each $q$, a tree $\Gamma_q$ compatible with $\eta^{(q)}$ that induces
a tree (rather than a forest) on $I \cup J$. Using this tree, one finds
that the left-hand side equals $\sigma_{Y_{IJ}}(\tilde{\eta}^{(q)})(f)$
where $\tilde{\eta}^{(q)}$ is derived from $\eta^{(q)}$ by setting the
entries with $(i,j) \not \in I \times J$ equal to infinity.  Then the
convergence follows by continuity of $\sigma_{Y_{IJ}}$ and the fact that
$\tilde{\eta}^{(q)} \to \eta$ for $q \to \infty$. 

The map $\Trop(X) \to \Rb^m, \xi \to \xi \odot (0,0,\ldots,0)^T=\tau$
is $\RR^m$-equivariant, and this implies that $\sigma$ is
$\RR^m$-equivariant. But the construction $\xi \mapsto \rho$ is not
$\RR^p$-equivariant. 
\end{proof}

\begin{re}
The proof above is not as satisfactory as the proof for Grassmannians of
two-spaces in Section~\ref{sec:Grassmannian}, which used the technique
of Proposition~\ref{prop:Smearing} to prove that the defined section is
independent of the decomposition $\xi=A\tau+\eta$ and hence equivariant.
We have tried to mimick the proof for the Grassmannian, but failed
because for suitable $\eta \in \Trop(Y^0)$ the set $T_\eta$ can have
dimension much larger than the expected dimension four. This implies
that the second requirement in Proposition~\ref{prop:Smearing} cannot be
satisfied. Of course, this does not rule out the existence of alternative
techniques for proving $\RR^{m+p}$-equivariance.
\end{re}

\section{$A$-discriminants} \label{sec:ADiscriminants}

Linear spaces smeared around by tori, as discussed in
Section~\ref{sec:Smearing}, arise in the study of $A$-discriminants from
\cite{Gelfand94}. Let $\phi:\Gm^m \to \Gm^n$ be a torus homomorphism
with corresponding integer $n \times m$-matrix $A$, and let $V$ be the
closure in $\AA^n$ of the image of $\phi$, a toric variety. The linear
action of $\Gm^m$ on $\AA^n$ gives rise to an action on the dual space
$(\AA^n)^\vee$, given by a torus homomorphism $\phi^\vee:\Gm^m \to \Gm^n$
corresponding to the matrix $-A$.

Let $Y \subseteq (\AA^n)^\vee$ be the annihilator of the tangent space
$T_{\phi(1)} V$. Since $A$, when regarded as a matrix over $K$, is the
derivative of $\phi$ at $1$, $Y$ is the orthogonal complement of the
column space of $A$. For $t \in \Gm^m$, $\phi(t)$ maps $T_{\phi(1)}V$
into $T_{\phi(t)}V$, hence we find that $\phi^\vee(t)$ maps $V$
into the annihilator of $T_{\phi(t)}V$.  Thus the variety $X$
defined as the Zariski closure of the union of these annihilators
equals $\overline{\phi^\vee(\Gm^m) \cdot Y}$.  This is known as the
Horn uniformisation of the {\em dual variety} of $V$. It was used in
\cite{Dickenstein07} to characterise $\Trop(X^0)$ as
\[ \Trop(X^0) = -A\RR^m + \Trop(Y^0), \]
where, of course, the minus sign is only a reminder of the contragredience
of the action of $\Gm^m$ on $(\AA^n)^\vee$ and can also be left out.
This leads to the following fundamental problem.

\begin{prb}
For which torus homomorphisms $\phi:\Gm^m \to \Gm^n$ does the map from
the analytification of the dual variety $X$ of $V=\overline{\im \phi}
\subseteq \AA^n$ to $\Trop(X)$ admit a continuous, $\RR^m$-equivariant
section into the subset $Z \subseteq \Trop(X)$ defined in
Section~\ref{sec:Torus}?
\end{prb}

We do not have any general results at this point. Instead, we now consider
the very special case of {\em Cayley's hyperdeterminant}, and we stay
away from zero coordinates.

\begin{ex}
Let $n=2^3$ and use coordinates $x_{ijk},\ i,j,k \in \{0,1\}$ on $\AA^8$.
Let $m=3\cdot 2$ and use coordinates $t_i,u_j,v_k,\ i,j,k \in \{0,1\}$ on
$\Gm^6$. Let $\phi$ be the map $(t,u,v) \mapsto (t_iu_jv_k)_{i,j,k}$. Then
$V^0$ is the variety of rank-one tensors of format $2 \times 2 \times
2$. The dual variety $X$ is a hypersurface whose defining equation is
Cayley's hyperdeterminant
\begin{align*} \Delta&=
x_{000}^2 x_{111}^2 + 
x_{001}^2 x_{110}^2 + 
x_{010}^2 x_{101}^2 + 
x_{100}^2 x_{011}^2 \\
&- 
2x_{000}x_{001}x_{110}x_{111} - 
2x_{000}x_{010}x_{101}x_{111} - 
2x_{000}x_{011}x_{100}x_{111}\\&-
2x_{001}x_{010}x_{101}x_{110} - 
2x_{001}x_{011}x_{110}x_{100} - 
2x_{010}x_{011}x_{101}x_{100} \\
&+
4x_{000}x_{011}x_{101}x_{110} + 
4x_{001}x_{010}x_{100}x_{111}.
\end{align*}
The tropical variety of $X$ is known explicitly (though we will not use
this knowledge): modulo its four-dimensional lineality space it is a
$3$-dimensional fan in $4$-space. Intersecting with a $3$-dimensional
sphere yields a $2$-dimensional spherical polyhedral complex, which
consists of two nested tetrahedra glued by quadrangles along corresponding
edges; see Figure~\ref{fig:Tetra}. This is the spherical complex of
the normal fan of the bipyramid over a tetrahedron from \cite[Section
2]{Huggins08}.

\begin{figure}
\includegraphics[scale=.5]{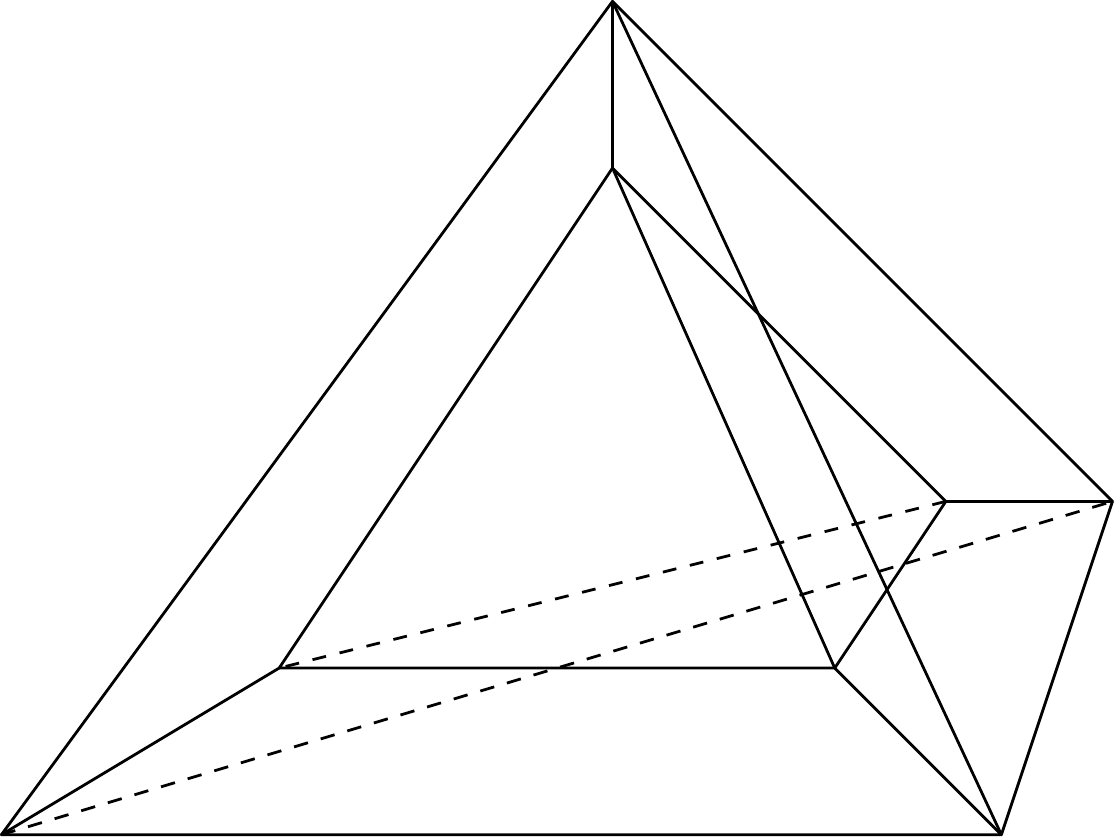}
\caption{The tropical variety of Cayley's hyperdeterminant has $8$
triangles and $6$ quadrangles.}
\label{fig:Tetra}
\end{figure}

The matrix $A$ sends $\tau=(\rho,\delta,\nu) \in \RR^6$ to the $2 \times 2
\times 2$-array with entries $(\rho_i+\delta_j+\nu_k)_{ijk}$. The kernel
of this map consists of vectors of the form $(a\one,b\one,c\one)$ with
$a+b+c=0$, so the column space $\im A$ has dimension $4$. It defines
the matroid on the vertices of the three-dimensional cube in which
independence is affine independence. Since the complement of any four
affinely independent vertices of the cube is again affinely independent,
this matroid is self-dual. So the dual matroid, which is the matroid of
the linear space $Y$, is the same matroid on $8$ elements. 

\begin{figure}
\begin{center}
\includegraphics{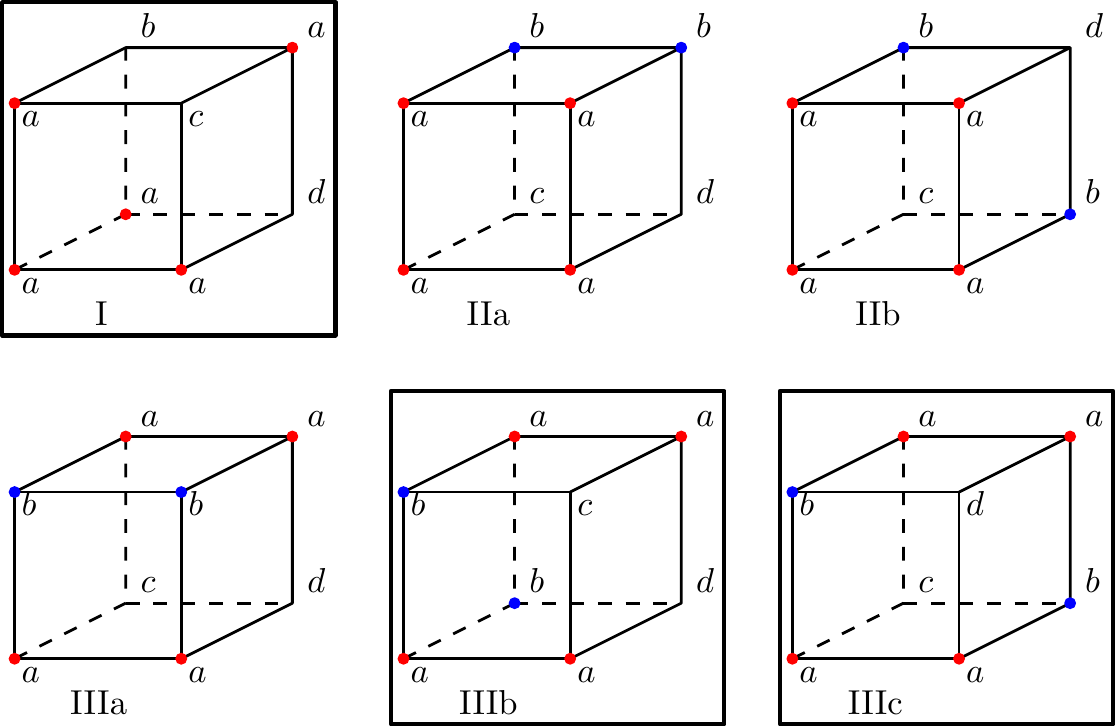}
\caption{The six orbits of maximal cones in $\Trop(Y^0)$, with $a \leq
b \leq c \leq d$.}
\label{fig:222}
\end{center}
\end{figure}

Up to symmetries of the cube, the seven-dimensional polyhedral
fan $\Trop(Y^0)$ has six maximal cones, and they are depicted in
Figure~\ref{fig:222}. Among these, the cones of type IIa, IIb, and IIIa
lie in $A\RR^6$ plus the union of the cones of type I, IIIb, IIIc. For
instance, take the array in type IIIa and add $(c-b)/2,0,(b-c)/2$ to the
positions with entries $b$,$a$,($c$ and $d$), respectively. The array
thus added lies in the column space of $A$, and the result is an array
in the boundary of type IIIb (with $b$ and $c$ replaced by $(b+c)/2$
and $d$ replaced by $d+(b-c)/2$).

Now let $C$ be a cone of type I, IIIb, or IIIc. Then the linear span of $C$
intersects $A\RR^6$ only in scalar multiples of the all-one array. This
follows from the fact that the span in $\RR^3$ of the differences of
vertices of the cube with the same label ($a$ or $b$) is all of $\RR^3$
(this is not true for the other types!). Thus on $A\RR^6 + C \subseteq
\Trop(X^0)$ we can define a section $\sigma_C$ into $\Trop(X^0)$
as follows: write $\xi$ as $A\tau + \eta$ with $\eta \in C$ and set
$\sigma(\xi):=\mu(\tau,\sigma_Y(\eta))$.  Note that, for any $c \in \RR$,
subtracting $(c\one,c\one,c\one)$ from $\tau$ and adding $3c$ times
the all-one array to $\eta$ yields the same value for $\sigma(\xi)$,
so that $\sigma$ is well-defined on $\AA \RR^6 + C$.

Next we verify that if $C'$ is a second cone of type I, IIIb, or IIIc,
then $\sigma_C$ and $\sigma_{C'}$ agree on the intersection $(A\RR^6 +
\sigma_C) \cap (A\RR^6 + C')$. This is immediate if
\[ (A\RR^6 + C) \cap (A\RR^6 + C') = A\RR^6 + (C \cap C') \]
as the recipes defining $\sigma_C$ and $\sigma_{C'}$ agree on the
right-hand side. For each choice of $C$ and $C'$, a vector witnessing that
the left-hand side is {\em strictly larger} than the right-hand side can
be found by solving a number of linear programs. If none of these linear
programs turns out to be feasible, then equality holds.  We have performed
this test for all choices of $C$ in the cones I, IIIb, IIIc, and $C'$ in
one of the orbits of these cones. Together with 
Proposition~\ref{prop:SmearingA} this proves the following theorem.

\begin{thm}
Let $X \subseteq K^{2 \times 2 \times 2}$ be the hypersurface defined by
Cayley's hyperdeterminant, equipped with the natural action of $\Gm^2
\times \Gm^2 \times \Gm^2$. Let $\Xan \to Z$ be the retraction defined
relative to this torus action. Then the surjection $Z^0 \to \Trop(X^0)$
has a continuous, $\RR^2 \times \RR^2 \times
\RR^2$-equivariant section. \hfill $\diamondsuit$
\end{thm}
\end{ex}

\bibliographystyle{alpha}
\bibliography{diffeq,draismajournal}

\newcommand{\etalchar}[1]{$^{#1}$}
\begin{thebibliography}{{Pay}09a}

\bibitem[AK06]{Ardila06}
Federico Ardila and Caroline~J. Klivans.
\newblock The {B}ergman complex of a matroid and phylogenetic trees.
\newblock {\em J. Comb. Theory, Ser. B}, 96(1):38--49, 2006.

\bibitem[Ber90]{Berkovich90}
Vladimir~G. Berkovich.
\newblock {\em Spectral theory and analytic geometry over non-archimedean
  fields}, volume~33 of {\em Mathematical Surveys and Monographs}.
\newblock American Mathematical Society, Providence, RI, 1990.

\bibitem[BG84]{Bieri84}
Robert Bieri and John~R.J. Groves.
\newblock The geometry of the set of characters induced by valuations.
\newblock {\em J. reine angew. Math.}, 347:168--195, 1984.

\bibitem[BJS{\etalchar{+}}07]{Bogart07}
Tristram Bogart, Anders~N. Jensen, David~E. Speyer, Bernd Sturmfels, and
  Rekha~R. Thomas.
\newblock Computing tropical varieties.
\newblock {\em J. Symb. Comp.}, 42(1-2):54--73, 2007.

\bibitem[BPR]{Baker11}
Matthew Baker, Sam Payne, and Joseph Rabinoff.
\newblock Nonarchimedean geometry, tropicalization, and metrics on curves.
\newblock Preprint, available from \verb+http://arxiv.org/abs/1104.0320+.

\bibitem[BR10]{Baker10}
Matthew Baker and Robert Rumely.
\newblock {\em {Potential theory and dynamics on the Berkovich projective
  line.}}
\newblock {Mathematical Surveys and Monographs 159. Providence, RI: American
  Mathematical Society (AMS)}, 2010.

\bibitem[CHW14]{Cueto14}
Maria~Angelica Cueto, Mathias H{\"a}bich, and Annette Werner.
\newblock Faithful tropicalization of the {G}rassmannian of planes.
\newblock {\em Math. Ann.}, 360(1--2):391--437, 2014.

\bibitem[DFS07]{Dickenstein07}
Alicia Dickenstein, Eva~Maria Feichtner, and Bernd Sturmfels.
\newblock Tropical discriminants.
\newblock {\em J. Am. Math. Soc.}, 20(4):1111--1133, 2007.

\bibitem[Dra08]{Draisma06a}
Jan Draisma.
\newblock A tropical approach to secant dimensions.
\newblock {\em J. Pure Appl. Algebra}, 212(2):349--363, 2008.

\bibitem[DSS05]{Develin05}
Mike {Develin}, Francisco {Santos}, and Bernd {Sturmfels}.
\newblock {On the rank of a tropical matrix.}
\newblock In {\em {Combinatorial and computational geometry}}, pages 213--242.
  Cambridge: Cambridge University Press, 2005.

\bibitem[EKL06]{Einsiedler06}
Manfred Einsiedler, Mikhail Kapranov, and Douglas Lind.
\newblock Non-archimedean amoebas and tropical varieties.
\newblock {\em J. reine angew. Math.}, 601:139--157, 2006.

\bibitem[GKZ94]{Gelfand94}
Israel~M. Gelfand, Mikhail~M. Kapranov, and Andrei~V. Zelevinsky.
\newblock {\em Discriminants, resultants, and multidimensional determinants}.
\newblock Mathematics: Theory \& Applications. Birkh\"auser, Boston, MA, 1994.

\bibitem[GRW]{Gubler14}
Walter Gubler, Joe Rabinoff, and Annette Werner.
\newblock Skeletons and tropicalizations.
\newblock Preprint, \verb+arxiv:1404.7044+.

\bibitem[HSYY08]{Huggins08}
Peter {Huggins}, Bernd {Sturmfels}, Josephine {Yu}, and Debbie~S. {Yuster}.
\newblock {The hyperdeterminant and triangulations of the 4-cube.}
\newblock {\em {Math. Comput.}}, 77(263):1653--1679, 2008.

\bibitem[{Iri}10]{Iriarte10}
Benjamin {Iriarte Giraldo}.
\newblock {Dissimilarity vectors of trees are contained in the tropical
  Grassmannian.}
\newblock {\em {Electron. J. Comb.}}, 17(1):research paper n6, 7, 2010.

\bibitem[{Man}12]{Manon12}
Christopher {Manon}.
\newblock {Dissimilarity maps on trees and the representation theory of
  $\mathrm{GL}_n(\mathbb{C})$.}
\newblock {\em {Electron. J. Comb.}}, 19(3):research paper p38, 12, 2012.

\bibitem[Mik06]{Mikhalkin06a}
Grigory Mikhalkin.
\newblock Tropical geometry and its applications.
\newblock In Marta Sanz-Sol\'e {\em et al.}, editor, {\em Proceedings of the
  international congress of mathematicians, Madrid, Spain, August 22--30, 2006.
  Volume II: Invited lectures.}, Z\"urich, 2006. European Mathematical Society.

\bibitem[MS01]{Maclagan12}
Diane Maclagan and Bernd Sturmfels.
\newblock {\em Introduction to {T}ropical {G}eometry}, volume 161 of {\em
  Graduate Studies in Mathematics}.
\newblock American Mathematical Society, Providence, RI, 201.

\bibitem[{Pay}09a]{Payne09}
Sam {Payne}.
\newblock {Analytification is the limit of all tropicalizations.}
\newblock {\em {Math. Res. Lett.}}, 16(2-3):543--556, 2009.

\bibitem[Pay09b]{Payne07}
Sam Payne.
\newblock Fibers of tropicalization.
\newblock {\em Math. Z.}, 262(2):301--311, 2009.

\bibitem[RGST05]{RichterGebert05}
J\"urgen Richter-Gebert, Bernd Sturmfels, and Thorsten Theobald.
\newblock First steps in tropical geometry.
\newblock In G.~L.~{\em et al.} Litvinov, editor, {\em Idempotent Mathematics
  and Mathematical Physics. Proceedings of the International Workshop, Vienna,
  Austria, February 3-10, 2003}, volume 377 of {\em Contemporary Mathematics},
  pages 289--317, Providence, RI, 2005. AMS.

\bibitem[Sch03]{Schrijver03}
Alexander Schrijver.
\newblock {\em Combinatorial Optimization. Polyhedra and Efficiency}.
\newblock Number~24 in Algorithms and Combinatorics. Springer, Berlin, etc.,
  2003.

\bibitem[SS04]{Speyer04}
David~E. Speyer and Bernd Sturmfels.
\newblock The tropical {G}rassmannian.
\newblock {\em Adv. Geom.}, 4(3):389--411, 2004.

\bibitem[YY07]{Yu07}
Josephine Yu and Debbie~S. Yuster.
\newblock Representing tropical linear spaces by circuits.
\newblock In {\em Proceedings of the 19th International Conference on Formal
  Power Series and Algebraic Combinatorics (FPSAC 2007), July 2007, Tianjin,
  China}, 2007.
\newblock Preprint available from \verb+http://arxiv.org/abs/math/0611579+.

\end{thebibliography}

\end{document}